%% file: main.tex
\RequirePackage{fix-cm}

\documentclass[11pt]{article}
\usepackage[letterpaper,margin=1in]{geometry}
\usepackage{authblk}

\providecommand{\titlerunning}[1]{}
\providecommand{\authorrunning}[1]{}
\providecommand{\institute}[1]{}

\usepackage{my_package}
\theoremstyle{plain}
\newtheorem{theorem}{Theorem}[section]
\newtheorem{lemma}[theorem]{Lemma}
\newtheorem{proposition}[theorem]{Proposition}
\newtheorem{corollary}[theorem]{Corollary}
\newtheorem{assumption}[theorem]{Assumption}
\theoremstyle{definition}
\newtheorem{definition}[theorem]{Definition}
\newtheorem{example}[theorem]{Example}
\theoremstyle{remark}

\crefname{assumption}{Assumption}{Assumptions}
\Crefname{assumption}{Assumption}{Assumptions}

\begin{document}

\title{Augmented Lagrangian methods for infeasible convex optimization problems and
diverging proximal-point algorithms
}


\author[1]{Roland Andrews}
\author[1]{Justin Carpentier}
\author[1]{Adrien Taylor}
\affil[1]{INRIA, Ecole Normale Sup\'erieure}
\affil[ ]{\texttt{roland.andrews@inria.fr}}



\date{}

\maketitle

\begin{abstract}
  This work investigates the convergence behavior of augmented Lagrangian methods (ALMs) when applied to convex optimization problems that may be infeasible.
   ALMs are a popular class of algorithms for solving constrained optimization problems.
   We demonstrate that, under mild assumptions, the sequences of iterates generated by ALMs converge to solutions of the ``closest feasible problem''.
    We establish progressively stronger convergence results, ranging from basic sequence convergence to more precise convergence rates, under a hierarchy of assumptions.

  This study leverages the classical relationship between ALMs and the proximal-point algorithm applied to the dual problem. 
  A key technical contribution is a set of concise results on the behavior of the proximal-point algorithm when applied to functions that may lack minimizers.
   These results pertain to its convergence in terms of its subgradients and of the values of the convex conjugate.
\end{abstract}

\section{Introduction} \label{sec:Introduction}

Constrained convex optimization problems arise naturally in numerous applications, spanning engineering design, machine learning, and economics, often as subroutines within larger optimization frameworks. While many algorithms assume and exploit the existence of feasible solutions, several applications (e.g., optimal control~\cite{scokaert1999feasibility} or optimization layers in deep learning~\cite{amos2017optnet,agrawal2019differentiable,bambade2024leveraging}) require robust behavior even when the feasible set is empty.
 
This work considers the standard convex optimization problem:
\begin{equation} \label{eq:convex_min_problem}
    \begin{array}{rl}
        \displaystyle \inf_{ x \in \X} & f(x)         \\[1mm]
        \text{s.t.}                    & C(x) \in \KK,
    \end{array}
\end{equation}
where the feasible set described by the constraint \(C(x)\in\KK\) may be empty. Formally, \(\X\) is a nonempty closed convex subset of a real Hilbert space,
 \(\Y\) is a real Hilbert space, and \(\KK\) is a nonempty closed convex subset of \(\Y\).
  We further assume that \(f: \X \to \R \cup \{+\infty\}\) is a closed, proper, convex function on \(\X\), and \(C: \X \to \Y\) is a mapping such that the graph of \( C - \KK\), defined as
\begin{equation} \label{def:CC} 
    \CC \changed{ \triangleq } \{ (x,s) \in \X \times \Y \; | \; s \in C(x) - \KK \},
\end{equation}
is a closed, convex\footnote{\added{This corresponds to the fact that \(C\) is (\(-\KK\))-convex in the sense of \cite[Definition 2.103]{bonnans2000perturbation}.}}, nonempty set.
 \(\KK\) need not be a cone; for instance, if \(C\) is a continuous affine function, \(\KK\) can be any nonempty closed convex set.

This framework encompasses many important classes of convex optimization problems, for example:
\begin{itemize}
    \item Quadratic programming (QP): \(\X = \R^n\), \(f: x \mapsto \tfrac{1}{2}x^\top H x + q^\top x\) (with \(H \succeq 0\)), \(\KK=\{0\}^{m_1} \times \R_-^{m_2}\), and \(C: x \mapsto (A x - a, Bx-b)\), where \(A \in \R^{m_1 \times n}\), \(B \in \R^{m_2 \times n}\), \(a \in \R^{m_1}\), and \(b \in \R^{m_2}\). This represents \(m_1\) equality constraints \(Ax=a\) and \(m_2\) inequality constraints \(Bx \leq b\).
    \item Semidefinite programming (SDP): \(\X = \mathbb{S}^n\) (the space of \(n \times n\) symmetric matrices), \(f: X \mapsto \langle C, X \rangle\), \(\KK = \{b\} \times \mathbb{S}^n_+\) (where \(b\in \R^m\) and \(\mathbb{S}^n_+\) is the cone of positive semidefinite matrices), and \(C: X \mapsto ((\langle A_1, X \rangle, \dots, \langle A_m, X \rangle), X)\). This represents constraints \(\langle A_i, X \rangle = b_i\) and \(X \succeq 0\).
    \item Convex second-order partial differential equations as illustrated in \Cref{example:second order elliptic PDE}. In this setting, \(\X\) and \(\Y\) are infinite-dimensional.
\end{itemize}

In this work, we study the behavior of the augmented Lagrangian method (ALM), a popular choice for addressing such problems. Historical references include~\cite{powell1969method,hestenes1969multiplier,rockafellar1976augmented}; \cite{bertsekas1982constrained} provides a detailed treatment.
 The ALM involves a partial dualization of the constraints (the constraint \(C(x)\in \KK\) is dualized, while \(x \in \X\) is not) and consists of alternating updates of the primal and dual variables. Thus, the ALM transforms a constrained optimization problem into a series of less constrained ones.
  The iterations of ALM are expressed as:
\begin{equation} \label{eq:ALM_pure}
    \begin{array}{rl}
        (x^{k+1} , y^{k+1} )  &\in \argmin\limits_{x \in \X, y \in \KK}  \left\{ f(x)  -\inner{\lambda^k}{C(x) - y} + \frac{\gamma_k}{2} \left\|C(x) - y \right\|^2 \right\} \\
        \lambda^{k+1} &= \lambda^k - \gamma_k  \left( C(x^{k+1}) - y^{k+1} \right),
    \end{array}
\end{equation}
where \(\gamma_k > 0\) is a sequence of positive real numbers known as penalty parameters, and \(\lambda^k \in \Y\) are the dual variables (Lagrange multipliers) for the constraint $C(x)\in \KK$. 
Iteration~\eqref{eq:ALM_pure} can often be rewritten and decomposed into more convenient forms in practical cases (see~\Cref{sec:ALM} for a detailed account of classical equivalent ALM reformulations).

\begin{example}[ALM for inequality-constrained convex optimization]. \label{example:ALM_for_inequality_constrained_convex_optimization}
    A specific instance of~\eqref{eq:convex_min_problem} is the nonlinear inequality-constrained convex optimization problem in finite dimensions:
    \begin{equation} \label{eq:inequality_constrained_problem}
        \begin{array}{rl}
            \displaystyle \min_{x \in \R^n} & f(x)         \\[1mm]
            \text{s.t.} & c_i(x) \leq 0, \quad i = 1, \dots, m,
        \end{array}
    \end{equation}
    which corresponds to~\eqref{eq:convex_min_problem} where \(\X = \R^n\), \(\KK = \R_-^{m} \) (the nonpositive orthant), and \( C: x \mapsto \left( c_1(x), \dots, c_{m}(x)\right)^\top \) is a vector of proper, closed, convex functions \(c_i(\cdot)\).
    The augmented Lagrangian method applied to this problem was first studied in~\cite{rockafellar1973dual}, where it is formulated as:
    \begin{equation} \label{eq:augmented_lagrangian_method_inequality_constraints}
        \begin{array}{rl}
            x^{k+1}  &\in \argmin\limits_{x \in \X}  \left\{ f(x) +  \sum_{i=1}^{m} \frac{1}{2\gamma_k} \left( \max\left(0, \lambda^k_i + \gamma_k c_i(x) \right)^2 - (\lambda^k_i)^2 \right) \right\} \\
            \lambda^{k+1}_i &= \max\left(0, \lambda^k_i + \gamma_k c_i(x^{k+1}) \right), \quad i = 1,\ldots,m.
        \end{array}
    \end{equation} 
    This is a particular case of~\eqref{eq:ALM_pure}, reformulated and simplified for this setting (details in \Cref{sec:ALM}). 
    Applying \Cref{thm:level boundedness locally uniformly implies lsc} in the present work guarantees that the iterates of ALM~\eqref{eq:augmented_lagrangian_method_inequality_constraints} converge to solutions of the closest feasible problem: \( \inf_{x \in \X} \{ f(x) \mid x \in \argmin_{\tilde x \in \mathrm{dom}(f)} \sum_{i=1}^m \max\left(0, c_i(\tilde x) \right)^2 \} \), provided \(f, c_1, \dots, c_m\) have no common recession direction.
\end{example}


\input{related_work.tex}

\input{contributions.tex}

\subsection{Notations and Preliminaries}

 \changed{We use standard notation from convex analysis. 
 In particular, $\lfloor\cdot\rfloor_+ := \max(0,\cdot)$ acts component-wise on vectors,
 $\delta_A$ denotes the indicator function of a set $A$,
 and $\R_-^m$ (resp.\ $\R_+^m$) denotes the nonpositive (resp.\ nonnegative) orthant. \(\KK^{\infty}\) denotes the cone generated by \(\KK\).
 } 
 
\bigskip

For convenience, we introduce a slack variable \(s\) into~\eqref{eq:convex_min_problem}, yielding the following equivalent reformulation
\begin{equation} \label{eq:convex_min_problem_rewritten}
    \begin{array}{rl}
        \displaystyle \inf_{x \in \X, s \in \Y} & f(x) + \delta_{\CC}(x,s) \\[1mm]
        \text{s.t.} & s = 0,
    \end{array}
\end{equation}
where the set \(\CC\) was defined in \eqref{def:CC}.
As detailed in \Cref{sec:ALM}, applying the IALM to problem~\eqref{eq:convex_min_problem_rewritten} or~\eqref{eq:convex_min_problem} is equivalent (i.e., it produces the same iterates up to a trivial change of variable). Therefore, there is no loss of generality in studying the IALM on~\eqref{eq:convex_min_problem_rewritten}.

\paragraph{Lagrangian and dual function.}
The Lagrangian function associated with problem~\eqref{eq:convex_min_problem_rewritten} is defined as
\begin{equation} \label{def:lagrangian function}
    L(x,s,\lambda) \triangleq f(x) + \delta_{ \CC}(x,s) - \inner{\lambda}{s},
\end{equation}
and we define the dual function of this problem as
\begin{equation} \label{eq:dual function}
    g(\lambda) \triangleq \underset{ \substack{ (x,s) \in \X \times \Y }}{ \text{inf }} L(x,s,\lambda) = \inf_{(x,s) \in \CC}   \left( f(x) - \inner{\lambda}{s} \right) .
\end{equation}
We define the negative dual function \(h \triangleq -g\) for notational convenience, which is convex and closed.

The augmented Lagrangian associated with problem~\eqref{eq:convex_min_problem_rewritten}, with a parameter \(\gamma >0\), is defined as
\begin{equation} \label{def:augmented_lagrangian_L_gamma}
    L_\gamma(x,s,\lambda) \triangleq f(x) + \delta_{ \CC}(x,s) - \inner{\lambda}{s} + \tfrac{\gamma}{2} \left\Vert s \right\Vert^2 .
\end{equation}
\added{ The augmented Lagrangian is used to define the inexact augmented Lagrangian method (IALM) as described in \Cref{algo:IALM}.}

\begin{algorithm}
    \caption{Inexact augmented Lagrangian method (IALM)} \label{algo:IALM}

    \textbf{Input: } Choose \(\lambda^0 \in \Y\), \((\gamma_k)_{k \in \N} \in \left(\R_{+}\backslash \{0\}\right)^\N \) and \((\varepsilon_k)_{k \in \N^*} \in \R_{+}^\N\).

    \textbf{Loop: for \(k = 0,1,2,\dots \) }

    \quad Find \(x^{k+1} \in \X\) and \(y^{k+1} \in \KK\) such that 
    
    \quad \quad \( \begin{multlined} L_{\gamma_k}(x^{k+1},s^{k+1},\lambda^k) - \inf_{(x,s) \in \CC}  L_{\gamma_k}(x,s,\lambda^k) \leq \tfrac{(\varepsilon_{k+1})^2}{2 \gamma_k}, \\ \text{where } s^{k+1} = C(x^{k+1}) - y^{k+1}. \end{multlined} \)

    \quad Set \(\lambda^{k+1} = \lambda^k - \gamma_k s^{k+1} \).

\end{algorithm}

For \( (x,s) \in \CC \) \added{(\(\CC\) is defined in \eqref{def:CC})}, we refer to \(s\) as the constraint violation vector. We can directly call \(C(x) - y\) a constraint violation vector for a given \(x \in \X\) and \(y \in \KK\), since then \((x, C(x) - y) \in \CC \).
We define the set of attainable constraint violations as 
\changed{
\begin{equation} \label{def:Scal}
\Scal  \triangleq \{ s \in \Y \;|\; \exists x \in \X \text{ such that } (x, s)\in \CC \}.
\end{equation}}
\added{\(\Scal\) is convex by the fact that it is the projection of the convex set \(\CC\) on the \(s\)-axis.}
The element of \( \cl(\Scal)\) (the closure of \(\Scal\)) with the smallest norm, denoted by \( \sBar \triangleq \arg\min_{s \in \cl(\Scal)} \|s\|^2 \), will play an important role.
 By a slight abuse of language, we call \(\sBar\) the smallest-norm constraint violation vector, although \(\sBar\) need not itself be an attainable constraint violation (i.e., \(\sBar\) may not belong to~\(\Scal\)). 

\paragraph{The value function.}
An important object is the shifted problem, where the constraint set \(\KK\) in~\eqref{eq:convex_min_problem} is shifted by a vector \(\tilde s \in \Y\):
\begin{equation} \label{eq:convex_min_problem_shifted}
    \begin{array}{rl}
        \displaystyle \nu(\tilde s) \triangleq \inf_{ x \in \X} & f(x)         \\[1mm]
        \text{s.t.}                    & C(x) \in \KK + \tilde s  .
    \end{array} \begin{array}{rl}
        \displaystyle = \inf_{ (x,s) \in \X \times \Y} & f(x)   + \delta_\CC(x,s)      \\[1mm]
        \text{s.t.}                    & s = \tilde s  .
    \end{array}
\end{equation}
In this context, \(\tilde s\) is called the shift; it is also sometimes referred to as a perturbation \cite[Chapter 28]{rockafellar1970convex}.
The function \(\nu: \Y \to \R \cup \{ \pm \infty \} \) is called the value function (or sometimes the perturbation function or optimal-value map) and is convex. 
Since \(\CC\) is nonempty, there always exists a shift \(\tilde s\) for which problem~\eqref{eq:convex_min_problem_shifted} is feasible (e.g., any \(\tilde s\) such that \((x_0, \tilde s) \in \CC\) for some \(x_0 \in \X\)), meaning that \(\nu\) is not identically \(+\infty\).
\changed{The case where the properness assumption does not hold (i.e., if \(\nu(\tilde s) = -\infty\) for some \(\tilde s\)) is trivial and is treated separately in \Cref{sec:case_where_the_shifted_problem_is_not_proper}.} 
Henceforth, we assume that \(\nu\) is proper, meaning that \(\nu(\tilde s) > -\infty\) for all \(\tilde s \in \Y\).

Since \(h(\lambda) = -g(\lambda) = \sup_{(x,s) \in \CC}  \inner{\lambda}{s} - f(x) =  \sup_{s \in \Y}  \inner{\lambda}{s} - \nu(s) \) is the convex conjugate of \(\nu\), it is proper and closed \cite[Theorem 12.2]{rockafellar1970convex}.

The \emph{conjugate dual} is the Fenchel conjugate of the negative dual function 
\begin{equation*}
    h^*(s) \triangleq \sup_{\lambda \in \Y} \inner{s}{\lambda} - h(\lambda) = \sup_{\lambda \in \Y} \inner{s}{\lambda} + g(\lambda) .
\end{equation*}
Hence, \(h^* = \nu^{**}\) is also the closure of the value function \(\nu\) \cite[Theorem 7]{rockafellar1974conjugate}.

\paragraph{Inexact proximal point algorithm.}
The proximal point algorithm for minimizing the function \(h\) involves iterating the proximal point operator, defined for a step size \(\gamma > 0\) as:
\begin{equation} \label{eq:prox-point-operator}
    \Prox_{\gamma h}(\lambda)  \triangleq  \argmin_{\mu \in \Y} \left( h(\mu) + \tfrac{1}{2\gamma}\left\Vert \lambda - \mu \right\Vert^2 \right) .
\end{equation}
The proximal operator is well-defined and single-valued when \(h\) is closed, convex \added{and proper} \cite{martinet1970breve}.
We refer to the method described in \Cref{algo:IPPA}, when applied to \(h\), as the inexact proximal point algorithm (IPPA). 
In this algorithm, the auxiliary sequences \( ( \lambda_\star^{k+1}, s_\star^{k+1} ) \) \added{are added for analytical convenience} and denote the iterates that would have been obtained from \( ( \lambda^{k}, s^{k} ) \) had there been no error (meaning had \(\varepsilon_k\) been equal to zero).

\begin{algorithm}
    \caption{Inexact proximal point algorithm (IPPA)}\label{algo:IPPA}
    \textbf{Input: } Choose \(\lambda^0 \in \Y\), \((\gamma_k)_{k \in \N} \in \left(\R_{+}\backslash \{0\}\right)^\N \) and \((\varepsilon_k)_{k \in \N^*} \in \R_{+}^\N\).

    \textbf{Loop: for \(k = 0,1,2,\dots\)}
    
    \quad Let \( \lambda_\star^{k+1} = \Prox_{\gamma_k h}(\lambda^k)\).

    \quad Let \(s^{k+1}_\star = \tfrac{ \lambda^k - \lambda_\star^{k+1} }{\gamma_k} \).

    \quad Find \( \lambda^{k+1} \in \Y \) such that  \( \left\Vert \lambda^{k+1} - \lambda_\star^{k+1} \right\Vert \leq \varepsilon_{k+1} \).

    \quad Let \(s^{k+1} = \tfrac{ \lambda^k - \lambda^{k+1} }{\gamma_k} \).
\end{algorithm}

\paragraph{Assumptions on the errors and step sizes.}

\changed{
In this paper, we make the following very mild assumptions.
\begin{assumption} \label{assumption:errors_and_step_sizes}
    The step sizes \((\gamma_k)_{k\in\N}\) are positive, the errors \((\varepsilon_k)_{k\in\N}\) are nonnegative, and the following assumptions hold:
\begin{itemize}[leftmargin=3.5em]
    \item[\mylabel{(A1)} ] \(\sum_{k=1}^\infty \varepsilon_k < \infty\).
    \item[\mylabel{(A2)} ] \(\tfrac{\varepsilon_{k+1}}{\gamma_k} \sum_{i=0}^k \gamma_i  = O\left( \tfrac{1}{\sqrt{\sum_{i=0}^k \gamma_i }} \right) \).
    \item[\mylabel{(A3)} ] \(\sum_{k=1}^\infty \left( \left( \varepsilon_{k} / \gamma_k \right)^2 \sum_{i=0}^{k-1} \gamma_i \right) < \infty\).
    \item[\mylabel{(A4)} ] \(\sum_{k=1}^{\infty} \left( ( \varepsilon_{k} / \gamma_k ) \sum_{i=1}^{k}   \gamma_{i-1} \right)  < \infty\).
    \item[\mylabel{(A5)} ] \(\sum_{k=0}^\infty \gamma_k = \infty\).
\end{itemize}
\end{assumption}
These assumptions ensure that the error decreases sufficiently fast to zero. For example, if the step sizes are constant \(\gamma_i = \gamma\),
 then the choice of error \( \varepsilon_k = \tfrac{1}{ ((k+1)\ln(k+1))^2} \) satisfies all the assumptions.
 Assumption (A5) is standard in the literature.
 Not all assumptions are used in all results but we group them together for readability.
}


\subsection{Reformulation of IALM as an inexact proximal point algorithm} \label{sec:IALM is IPPA}

\added{In this subsection, we identify the relationship between the IALM and quantities of the dual problem, which will be used in \Cref{sec:convergence_ialm} to translate the results of \Cref{sec:IPPA} into results on IALM.}
It is well known that the inexact augmented Lagrangian method can be interpreted as an inexact proximal point algorithm on the dual function \cite[Proposition 6]{rockafellar1970convex}. We recall this result here for completeness\deleted{ and because some elements of the proof are used later in this work}. 
\begin{proposition} \cite[Proposition 6]{rockafellar1976augmented} \label{prop:approx aug Lagrangian is approx prox point}
   Consider the convex optimization problem~\eqref{eq:convex_min_problem}. For any sequence \((\lambda_\star^{k+1}, s_\star^{k+1}, \lambda^{k+1} , s^{k+1})_{k\in\N}\) generated by IALM (\Cref{algo:IALM}) on \eqref{eq:convex_min_problem} with positive penalty parameters \((\gamma_k)_{k\in\N} \) and nonnegative errors \((\varepsilon_k)_{k\in\N}\), 
    \( (\lambda^k , s^k)_{k\in \N}\) is a valid sequence of iterates of IPPA (\Cref{algo:IPPA}) on the negative dual function \(h\) with step sizes \((\gamma_k)_{k\in\N} \) and errors \((\varepsilon_k)_{k\in\N}\).
\end{proposition}

\begin{proof}
First, let us note that, for any \(\mu \in \Y\), the following holds 
\begin{align}
    \inf_{(x,s)\in\CC} \; L_{\gamma_k}( x, s ,\mu) & = \inf_{(x,s)\in\CC} \; f(x) - \inner{\mu}{s} + \tfrac{\gamma_k}{2} \left\Vert s \right\Vert^2 \\
    & = \inf_{(x,s)\in\CC} \; \sup_{\lambda \in \Y} \;  f(x) -  \inner{\lambda}{s} - \tfrac{1}{2\gamma_k} \left\Vert  \mu - \lambda \right\Vert^2 \nonumber \\
    & =  \sup_{\lambda \in \Y} \; \inf_{(x,s)\in\CC} \;  f(x) - \inner{\lambda}{s} - \tfrac{1}{2\gamma_k} \left\Vert  \mu - \lambda \right\Vert^2 \nonumber \\
    & = \sup_{\lambda \in \Y} \; -h(\lambda)  - \tfrac{1}{2\gamma_k} \left\Vert  \mu - \lambda \right\Vert^2,  \label{eq:inf_sup_swap}
\end{align}
where we used \cite[Theorem 6]{rockafellar1971saddle} to swap \(\inf\)  and \(\sup \) which is applicable because for any \(s\in \Y\) the concave function  \(\lambda \to  f(x) - \inner{\lambda}{s} - \tfrac{1}{2\gamma} \left\Vert  \mu - \lambda \right\Vert^2 \)  has bounded level sets. 

Using \eqref{eq:inf_sup_swap}, observe that an exact ALM step corresponds to an exact proximal point step on the dual:
\begin{align}
    \inf_{(x,s)\in\CC} \; L_{\gamma_k}( x, s ,\lambda^k) & = \sup_{\lambda \in \Y} -h(\lambda)  - \tfrac{1}{2 \gamma_k} \left\Vert  \lambda^k - \lambda \right\Vert^2 \nonumber                                                             \\
    & = -h(\lambda_\star^{k+1})  - \tfrac{1}{2 \gamma_k} \left\Vert  \lambda^k - \lambda_\star^{k+1} \right\Vert^2 \label{eq:equal1_proof_IALMisIPPA}
\end{align}

We now lower bound the value of the augmented Lagrangian when it is approximately minimized. For any \(\mu \in \Y\) we have 
\begin{align}
    L_{\gamma_k}(x^{k+1},s^{k+1},\lambda^k)  - \inner{\mu - \lambda^k}{s^{k+1}} & = f(x^{k+1}) - \inner{\mu}{s^{k+1}} + \tfrac{\gamma}{2} \left\Vert s^{k+1} \right\Vert^2 \nonumber \\ 
        &= L_{\gamma_k}(x^{k+1},s^{k+1}, \mu ) \nonumber      \\
    & \geq \inf_{(x,s)\in\CC} \; L_{\gamma_k}( x, s ,\mu) \nonumber  \\
    &= \begin{multlined}[t]
      \sup_{\lambda \in \Y} \; -h(\lambda)  - \tfrac{1}{2\gamma_k} \left\Vert  \mu - \lambda \right\Vert^2 \\ \text{(using \eqref{eq:inf_sup_swap})}  
      \end{multlined}    \nonumber \\
    & \geq -h(\lambda_\star^{k+1})  - \tfrac{1}{2\gamma_k} \left\Vert  \mu - \lambda_\star^{k+1} \right\Vert^2  \label{eq:inequ1_proof_IALMisIPPA} .
\end{align}

Taking the difference between \eqref{eq:inequ1_proof_IALMisIPPA} and \eqref{eq:equal1_proof_IALMisIPPA}, we get, for any \(\mu \in \Y \):
\begin{multline*}
  L_{\gamma_k}(x^{k+1},s^{k+1},\lambda^k)  - \inf_{(x,s)\in\CC}\; L_{\gamma_k}( x, s ,\lambda^k)    \\
  \geq  - \tfrac{1}{2\gamma_k} \left\Vert  \mu - \lambda_\star^{k+1} \right\Vert^2 + \inner{\mu - \lambda^k}{s^{k+1}} + \tfrac{1}{2 \gamma_k} \left\Vert  \lambda^k - \lambda_\star^{k+1} \right\Vert^2 .
\end{multline*}

Hence, the optimal choice \(\mu = \lambda_\star^{k+1} - \lambda^{k+1} + \lambda^{k} \) with \( s^{k+1} = \tfrac{\lambda^k - \lambda^{k+1}}{\gamma_k} \) yields:
\begin{equation*}
    L_{\gamma_k}(x^{k+1},s^{k+1},\lambda^k)  - \inf_{(x,s)\in\CC}\; L_{\gamma_k}( x, s ,\lambda^k) \geq \tfrac{\left\Vert  \lambda^{k+1} - \lambda_\star^{k+1} \right\Vert^2}{2\gamma_k} .
\end{equation*}

Therefore if \( L_{\gamma_k}(x^{k+1},s^{k+1},\lambda^k)  - \inf_{(x,s)\in\CC}\; L_{\gamma_k}( x, s ,\lambda^k) \leq \tfrac{ (\varepsilon_{k+1} )^2}{2 \gamma_k}\) then \(\left\Vert  \lambda^{k+1} - \lambda_\star^{k+1} \right\Vert \leq  \varepsilon_{k+1}  \). This shows that the sequence \( (\lambda^k)_{k\in\N}\) is indeed following the iterative rules described in \Cref{algo:IPPA}.  
\end{proof}

A direct consequence of \Cref{prop:approx aug Lagrangian is approx prox point} is that any property shown on iterates of \Cref{algo:IPPA} also applies to iterates from \Cref{algo:IALM}.

\added{Recall that \(\Scal\) is defined in~\eqref{def:Scal} and \(\sBar\) is defined right below~\eqref{def:Scal}.} The following lemma allows us to define \(\sBar\) completely in terms of the dual function as \(\sBar = \argmin \{\Vert s \Vert \; |\; s \in \cl( \range (\partial h))\}\) without referring to the primal set \(\Scal\). 
\begin{lemma} \label{lem:set_inclusions} 
    \( \cl(\Scal) = \cl( \dom (\partial h^*)) = \cl( \range (\partial h))\).
\end{lemma}
 
\begin{proof}
The Brondsted-Rockafellar theorem \cite[Theorem 2]{brondsted1965subdifferentiability} states that \\ \( \cl(\dom(\partial h^*)) = \cl(\dom(h^*)) \). Since \(h^*\) is the closure of \(\nu\) \added{ \cite[Theorem 7]{rockafellar1974conjugate}}, we also have \( \cl(\dom( h^*)) = \cl(\dom(\nu)) \). From the definition of \(\nu\), it is straightforward that \(\Scal = \dom(\nu) \). These three equalities combined lead to the desired result \( \cl(\Scal) = \cl( \dom (\partial h^*))\). 
Moreover, \(\dom (\partial h^*) = \range (\partial h)\) is well known when \(h\) is closed \cite[Theorem 23.5]{rockafellar1970convex}. 
\end{proof}

 The next proposition establishes a link between the primal objective values \(f(x^k)\) of the IALM algorithm and the exact dual conjugate values \(h^*(s_\star^k) \). 

 \begin{proposition} \label{prop:h*(sk) close to f(xk)}
     Let \((\lambda_\star^{k+1}, s_\star^{k+1}, \lambda^{k+1} , s^{k+1})_{k\in\N}\) be a sequence generated by an IALM (\Cref{algo:IALM}) associated with Problem \eqref{eq:convex_min_problem} \changed{with errors \((\varepsilon_k)_{k\in\N^*}\) and step sizes \((\gamma_k)_{k\in\N}\) satisfying \Cref{assumption:errors_and_step_sizes}}. 
      If the sequence \( (s^k)_{k\in\N} \) is bounded then \( | f(x^{k+1}) - h^*(s_\star^{k+1}) | = O\left(\tfrac{1}{\sqrt{\sum_{i=0}^k \gamma_i}} \right) \).
 \end{proposition}
 
 \begin{proof}
We have \( \lambda^{k+1}_\star = \argmin_{\mu\in\Y} \left( h(\mu) + \tfrac{1}{2\gamma_k} \Vert \lambda^k - \mu\Vert^2\right)\).
Therefore, \( s_\star^{k+1} = \tfrac{\lambda^k - \lambda^{k+1}_\star}{\gamma_k} \in \partial h(\lambda_\star^{k+1})\).
Therefore, using the Fenchel-Young equality, \( h(\lambda_\star^{k+1}) + h^*(s_\star^{k+1}) = \inner{s_\star^{k+1}}{\lambda_\star^{k+1}} \), and \eqref{eq:equal1_proof_IALMisIPPA}, yields:
\begin{align} 
    \inf_{(x,s)\in\CC} L_{\gamma_k}( x, s ,\lambda^k) & = -h(\lambda_\star^{k+1}) - \tfrac{1}{2 \gamma_k} \left\Vert \lambda_\star^{k+1} - \lambda^k \right\Vert^2    \nonumber          \\
    & = h^*(s_\star^{k+1}) - \inner{s_\star^{k+1}}{\lambda^{k}}  + \tfrac{\gamma_k}{2 } \left\Vert s_\star^{k+1} \right\Vert^2 . \nonumber
\end{align} 
We use this equality in the following expression of the error in the IALM:
\begin{align}
    &L_{\gamma_k}(x^{k+1}, s^{k+1}, \lambda^k) - \inf_{(x,s)\in\CC} L_{\gamma_k}( x, s ,\lambda^k) \nonumber \\
    &= \begin{multlined}[t]
      f(x^{k+1}) - \inner{s^{k+1}}{\lambda^k} +  \tfrac{\gamma_k}{2} \left\Vert s^{k+1} \right\Vert^2 - \quad \\ \left(h^*(s_\star^{k+1}) - \inner{s_\star^{k+1}}{\lambda^{k}}  + \tfrac{\gamma_k}{2 } \left\Vert s_\star^{k+1} \right\Vert^2\right)  
      \end{multlined} \nonumber \\
    &=  f(x^{k+1}) - h^*(s_\star^{k+1}) + \inner{s_\star^{k+1} - s^{k+1}}{\lambda^k} + \tfrac{\gamma_k}{2} \left( \left\Vert s^{k+1} \right\Vert^2 - \left\Vert s_\star^{k+1} \right\Vert^2 \right) \nonumber \\
    &= \added{ \begin{multlined}[t] f(x^{k+1}) - h^*(s_\star^{k+1}) + \inner{ \tfrac{\lambda^{k+1} - \lambda_\star^{k+1}}{\gamma_k} }{\lambda^k} \\ \quad + \tfrac{1}{2\gamma_k} \left( \left\Vert \lambda^k - \lambda^{k+1} \right\Vert^2 - \left\Vert \lambda^{k} - \lambda_\star^{k+1} \right\Vert^2 \right) \end{multlined} }  \nonumber \\
    &= f(x^{k+1}) - h^*(s_\star^{k+1}) + \tfrac{1}{2\gamma_k} \left( \Vert \lambda^{k+1} \Vert^2 - \Vert \lambda_\star^{k+1} \Vert^2 \right). \label{eq:augmented_lagrangian_dual_conjugate_inequality}
\end{align}
Furthermore, the inexactness condition from \Cref{algo:IALM} gives:
\begin{equation*} 
    0 \leq L_{\gamma_k}(x^{k+1}, s^{k+1}, \lambda^k) - \inf_{(x,s)\in\CC} L_{\gamma_k}( x, s ,\lambda^k) \leq \tfrac{\left(\varepsilon_{k+1}\right)^2}{2\gamma_k}
\end{equation*}
which, using \eqref{eq:augmented_lagrangian_dual_conjugate_inequality} becomes:
\begin{equation} \label{eq:f(xk) - h*(s_k) bound}
     \vert f(x^{k+1})   - h^*(s_\star^{k+1}) \vert \leq \tfrac{1}{2 \gamma_k} \left\vert \left( \Vert \lambda^{k+1} \Vert^2 - \Vert \lambda_\star^{k+1} \Vert^2 \right) \right\vert + \tfrac{(\varepsilon_{k+1})^2}{2\gamma_k} . 
\end{equation}
\changed{Choose an upper bound $M$ such that $ \Vert s^{i+1} \Vert \leq M$ for all $i \in \N$. Then}
\begin{align}
    \tfrac{1}{2\gamma_k} \left| \left\Vert \lambda^{k+1} \right\Vert^2 - \left\Vert  \lambda_\star^{k+1}  \right\Vert^2  \right| & = \tfrac{1}{2\gamma_k} \left|  \inner{\lambda^{k+1} - \lambda_\star^{k+1}}{\lambda^{k+1} + \lambda_\star^{k+1}} \right|   \nonumber     \\
                & \leq \tfrac{1}{2\gamma_k} \Vert\lambda^{k+1} - \lambda_\star^{k+1} \Vert \Vert \lambda^{k+1} + \lambda_\star^{k+1} \Vert \nonumber      \\
                & \leq \tfrac{1}{2\gamma_k} \varepsilon_{k+1}  \Vert \lambda^{k+1} + \lambda_\star^{k+1} \Vert    \nonumber   \\
                & \leq \tfrac{\varepsilon_{k+1}}{2\gamma_k} \left( \Vert 2\lambda^{k+1} \Vert  + \Vert   \lambda_\star^{k+1} - \lambda^{k+1} \Vert \right)    \nonumber  \\
                & \leq \tfrac{\varepsilon_{k+1}}{2\gamma_k} \left( 2 \left\Vert  \lambda^0 - \sum_{i=0}^{k} \gamma_i s^{i+1}\right\Vert  + \left\Vert \lambda_\star^{k+1} - \lambda^{k+1}   \right\Vert \right)    \nonumber  \\
                & \leq \begin{multlined}[t] \tfrac{\varepsilon_{k+1}}{\gamma_k} \left(\left\Vert \lambda^0 \right\Vert + M \sum_{i=0}^k \gamma_i + \tfrac{\varepsilon_{k+1}}{2} \right) .\end{multlined} \label{eq:f(xk) - h*(s_k) bound 2}
\end{align}
Combining \eqref{eq:f(xk) - h*(s_k) bound 2} and  \eqref{eq:f(xk) - h*(s_k) bound}, we obtain  
 \begin{align*}
     \left| f(x^{k+1}) - h^*(s_\star^{k+1}) \right| &\leq \tfrac{\varepsilon_{k+1}}{\gamma_k} \left(\left\Vert \lambda^0 \right\Vert + M \sum_{i=0}^k \gamma_i + \tfrac{\varepsilon_{k+1}}{2}  \right) + \tfrac{(\varepsilon_{k+1})^2}{2\gamma_k} \\ 
     &= O\left( \tfrac{\varepsilon_{k+1}}{\gamma_k} \sum_{i=0}^k \gamma_i \right) . 
 \end{align*}
\added{Where we used \(\varepsilon_{k+1} \to 0 \) due to \ref{(A1)}.}
           
If \ref{(A2)} holds, then  \( | f(x^{k+1}) - h^*(s_\star^{k+1}) | = O\left(\tfrac{1}{\sqrt{\sum_{i=0}^k \gamma_i}} \right) \).
 \end{proof}

 In the preceding proposition, the condition that the sequence \( (s^k)_{k\in\N} \) is bounded is actually not restrictive because, as we see below in \Cref{lem:IPPA_properties}, the sequence \( (s^k)_{k\in\N} \) is always bounded when \changed{\Cref{assumption:errors_and_step_sizes} holds}.


\input{sec_IPPA.tex}

\input{sec_convergence_of_IALM.tex}

\input{example_1.tex}

\input{conclusion.tex}


%
%

\bibliographystyle{plain}
\bibliography{references}   

\newpage

\input{appendix}

\end{document}

%% file: related_work.tex
\subsection{Related work} \label{sec:related_works}

When a convex optimization problem~\eqref{eq:convex_min_problem} is feasible \added{and satisfies some constraint qualification}, the inexact augmented Lagrangian method (IALM) (\Cref{algo:IALM}) provides a convenient way to approximate its solutions. 
In particular, it asymptotically converges to solutions of problem~\eqref{eq:convex_min_problem} by approximately solving a sequence of less constrained convex problems \cite{rockafellar1976augmented}. If the inner convex problems can be solved efficiently, the IALM is an effective method for approximating a solution to~\eqref{eq:convex_min_problem} under minimal assumptions. However, the algorithm's behavior in the infeasible setting has garnered more interest only recently.

\paragraph{Infeasibility detection.}
A common strategy in solvers for handling potential infeasibility is to design rules and algorithms for its detection. Implementations of specific algorithms often include infeasibility detection routines tailored to particular problem types. For example, QPALM, a proximal augmented Lagrangian method for (nonconvex) quadratic programs \cite{hermans2022qpalm}, incorporates such routines with heuristics.
\cite{armand2019augmented} investigates infeasibility detection for equality-constrained optimization using a primal-dual augmented Lagrangian method where the objective function is scaled by an additional parameter. When the problem is infeasible, this scaling parameter converges to zero, and the algorithm tends to minimize constraint violation. They assume convergence of the sequences, smoothness of the involved functions, and positive definiteness of the Hessian of the constraint norm at the limit point. Under these assumptions and with a specific parameter choice, linear convergence to an infeasible stationary point is achieved.

\paragraph{The augmented Lagrangian method as a penalty method.}
Some works modify the augmented Lagrangian method to simplify its analysis in the infeasible setting. In \cite{birgin2015optimality,gonccalves2015augmented},
 the multipliers are constrained to a bounded set.
  The penalty parameters are driven to infinity, rendering the multipliers asymptotically irrelevant and causing the algorithm to resemble a penalty method.
 This approach, however, shifts much of the computational difficulty to the ALM subproblems, as large penalty parameters typically lead to ill-conditioned subproblems.

\paragraph{Convergence of the augmented Lagrangian method on infeasible problems.}
The specific question of the augmented Lagrangian method's behavior in the case of infeasibility has been studied for quadratic programming (QP) in \cite{chiche2016augmented}. 
They provide rules for dynamically choosing penalty parameters to achieve any desired linear convergence rate. They prove that,
 for QPs, the ALM iterates converge to solutions of the closest feasible problem (minimizing $f$ over the set of points with minimum norm of the constraint violation). 
\cite{dai2023augmented} studies the ALM in the infeasible case within a more general convex setting similar to ours~\eqref{eq:convex_min_problem} with an approach that resembles that of \cite{chiche2016augmented}.
 Their results require several hypotheses: differentiability of all functions involved, subdifferentiability of the value function at the minimal shift, and the existence of a converging subsequence of iterates. 
 Under these hypotheses, they prove the existence of a subsequence that asymptotically satisfies KKT-like conditions for the exact augmented Lagrangian method (ALM).
  However, their results are often not applicable due to restrictive hypotheses.
   We illustrate in \Cref{example:QCQP_example} how even simple examples might not satisfy their hypotheses and thus fall outside the scope of their analysis. 

\paragraph{Infeasible and inexact proximal point algorithms.}
In \Cref{sec:IPPA}, we study an inexact proximal point algorithm (IPPA) when it is applied to a function that may have no minimizers. 
Our definition of IPPA (\Cref{algo:IPPA}) corresponds to the dual algorithm of IALM \cite{rockafellar1976augmented}.
Different types of inexact proximal point algorithms are studied in the literature; see \cite{solodov2001unified} for an analysis of various ways to define inexact proximal point iterations and their convergence properties when the problem is feasible.
It has been shown that the iterates of the (exact) proximal point algorithm diverge when there is no solution to the problem being studied, but the function values along the iterates still converge to its infimum~\cite{lemaire1992convergence}.

\paragraph{Value function.}
The value function \(\nu\) represents the optimal value of the optimization problem when the constraints are shifted (or perturbed) by a parameter \(s\), i.e., \( \nu: s \mapsto \inf \{ f(x) \mid x\in \X, \; C(x)\in \KK + s \} \).
The strength of some of our results on IALM convergence depends on whether the value function is lower-semicontinuous, subdifferentiable, or neither.
 The connection between the subdifferentiability of the value function and the existence of Kuhn-Tucker vectors for the perturbed problem, as well as the link between the lower semicontinuity of the value function and strong duality, is well-known \cite[Chapter 29 and 30]{rockafellar1970convex}.
  The value function under more general types of perturbations has also been studied in the nonconvex case. 
   For more general perturbations of the value function, see e.g. \cite{hogan1973directional,gauvin1982differential}.
    An extensive and comprehensive study of these questions in the infinite-dimensional case with general constraints can be found in \cite{bonnans2000perturbation}.

%% file: contributions.tex
\changed{ \subsection{Contributions}
We establish the following results:
\begin{itemize}
    \item For the inexact proximal point algorithm (IPPA) applied to a convex function \(h\) that may lack a minimizer (\Cref{sec:IPPA}), 
    we characterize necessary and sufficient conditions for the convergence of the iterates, the subgradient iterates, and the convex-conjugate iterates generated by the algorithm, both with and without convergence rates.
    \item For the inexact augmented Lagrangian method (IALM) applied to a convex optimization problem that may lack a feasible solution (\Cref{sec:convergence_ialm}) or fail to satisfy a constraint qualification,
    we establish convergence of the objective function values, the constraint violations, and the primal iterates, both with and without convergence rates.
    We also provide sufficient conditions under which the limits of these quantities correspond to solutions of the closest feasible problem.
\end{itemize}}

%% file: sec_IPPA.tex
\section{Inexact Proximal Point Algorithm} \label{sec:IPPA}
This section provides convergence results for the inexact proximal point algorithm (IPPA), as defined in \Cref{algo:IPPA}, when applied to a closed, proper,
convex function \(h\). The function \(h\) does not necessarily have a minimizer and is not assumed to be bounded from below.
Proofs in this section were facilitated by the performance estimation approach (PEP) \cite{drori2014performance,taylor2017exact}.

\begin{lemma} \label{lem:IPPA_properties}
Let \(h\) be a closed, proper, convex function.
 Let the sequence \\ \((\lambda_\star^{k+1}, s_\star^{k+1}, \lambda^{k+1} , s^{k+1})_{k\in\N}\) be generated by an IPPA (\Cref{algo:IPPA}) on \(h\) 
 \changed{with errors \((\varepsilon_k)_{k\in\N^*}\) and step sizes \((\gamma_k)_{k\in\N}\) satisfying \Cref{assumption:errors_and_step_sizes}}.
  The sequences \( (s_\star^k)_{k \in \N} \) and \( (s^k)_{k\in\N} \) are bounded.
\end{lemma} 

\begin{proof}
Since \(s_\star^{k} \in \partial h(\lambda_\star^k)\) for all \( k \in \N^*\), by convexity of \(h\), we have for \(k \ge 1\):
\begin{align}
    && 0 &\leq \inner{s_\star^{k} - s_\star^{k+1}}{\lambda_\star^k - \lambda_\star^{k+1}} \nonumber \\
    \Rightarrow&& 0& \leq \inner{s_\star^{k} - s_\star^{k+1}}{\lambda_\star^k - (\lambda^k - \gamma_{k} s_\star^{k+1})} \nonumber \\
    \Rightarrow&& 2 \inner{s_\star^{k+1} - s_\star^{k}}{s_\star^{k+1}} &\leq \tfrac{2}{\gamma_{k}}\inner{s_\star^{k} - s_\star^{k+1}}{\lambda_\star^k - \lambda^k}  \text{ (dividing by $\gamma_k > 0$)}\nonumber \\
    \Rightarrow&&  2\inner{s_\star^{k+1} - s_\star^{k}}{s_\star^{k+1}} &\leq \begin{multlined}[t] \tfrac{\left\Vert \lambda_\star^k - \lambda^k \right\Vert^2}{(\gamma_{k})^2} + \left\Vert s_\star^{k} - s_\star^{k+1} \right\Vert^2  \\ \text{ (using $2ab \leq a^2+b^2$ for the RHS)} \end{multlined}\nonumber\\
    \Rightarrow&& \left\Vert s_\star^{k+1} \right\Vert^2 - \left\Vert s_\star^{k} \right\Vert^2 &\leq \left( \tfrac{\varepsilon_{k}}{\gamma_{k}} \right)^2  \label{eq:majoration_s_star_k} \quad \text{(using $\Vert \lambda_\star^k - \lambda^k \Vert \leq \varepsilon_k$)} .
\end{align}
Summing these inequalities for \(k\) from \(1\) to \(N-1\) yields \( \left\Vert s_\star^{N} \right\Vert^2 \leq \left\Vert s_\star^{1} \right\Vert^2 + \sum_{k=1}^{N-1} \left( \varepsilon_k / \gamma_{k} \right)^2\).
Furthermore, assumption \ref{(A3)} implies that \(\sum_{k=0}^\infty \left( \varepsilon_k / \gamma_{k} \right)^2 < \infty \), which means that the sequence \( \left( s_\star^{k} \right)_{k\in \N}\) is bounded.
Assumptions \ref{(A5)} and \ref{(A2)} imply that \(\left( \varepsilon_{k+1} / \gamma_k \right)_{k\in \N}\) converges to zero, 
 hence \( \left( s^{k} \right)_{k\in \N}\) is also bounded due to the relation \( \Vert s_\star^{k+1} - s^{k+1} \Vert \added{ = \Vert \lambda^{k+1} - \lambda_\star^{k+1} \Vert / \gamma_k} \leq \varepsilon_{k+1} / \gamma_k \).
\end{proof}

In what follows, we denote by \( M_s \) a common strict upper bound for the norms of the iterates \( s_\star^{k}\) and \( s^{k}\), i.e.,

\begin{equation} \label{eq:M_s_definition}
    \changed{\forall k \in \N^*, \; M_s > \left\Vert s_\star^k \right\Vert \text{ and } M_s > \left\Vert s^k \right\Vert.}
\end{equation}
The next proposition demonstrates that the sequence \(s_\star^k\) actually converges to the element of minimum norm in \( \cl(\range(\partial h)) \).
A similar result in the more general setting of monotone inclusions is provided in~\cite{reich1977infinite} using a different approach \changed{but is not strong enough for our purposes.}

\begin{proposition} \label{prop:sk_convergence}
    Let \(h\) be a closed proper convex function. Let the sequence \((\lambda_\star^{k+1}, s_\star^{k+1}, \lambda^{k+1} , s^{k+1})_{k\in\N}\) be generated by an IPPA (\Cref{algo:IPPA})
    on \(h\) \changed{with errors \((\varepsilon_k)_{k\in\N^*}\) and step sizes \((\gamma_k)_{k\in\N}\) satisfying \Cref{assumption:errors_and_step_sizes}}. 
    \newline Let  \(\sBar = \argmin_{s \in \cl(\range(\partial h) )} \|s\|^2\) be the smallest-norm element in the closure of the range of \(\partial h\).

    Then, \( s_\star^k \to \sBar \) and \( s^k \to \sBar \).
    Furthermore, if \(h^*(\sBar) < \infty\), then the following convergence rates hold:
    \[
\begin{array}{ll}
\text{(a)}\quad \|s_\star^N\|^2 - \|\bar{s}\|^2 = O\left(\frac{1}{\sum_{i=0}^{N-1} \gamma_i}\right)
\quad &
\text{(b)}\quad \|s^N\|^2 - \|\bar{s}\|^2 = O\left(\frac{1}{\sum_{i=0}^{N-1} \gamma_i}\right)
\\[1em]
\text{(c)}\quad \|s_\star^N - \bar{s}\| = O\left(\frac{1}{\sqrt{\sum_{i=0}^{N-1} \gamma_i}}\right)
\quad &
\text{(d)}\quad \|s^N - \bar{s}\| = O\left(\frac{1}{\sqrt{\sum_{i=0}^{N-1} \gamma_i}}\right)
\end{array}
\]
\end{proposition}

\begin{proof}
    Let \(\tilde s\) be any point in \(\Scal\) such that \( \Vert \tilde s \Vert \leq M_s\) and \( h^*(  \tilde s ) < \infty\).
    For \(k\in \N\), by the convexity of \(h\) and the fact that \(s_\star^{k+1} \in \partial h(\lambda_\star^{k+1})\), we have:
    \begin{align}
        && \inner{s_\star^{k+1}}{\lambda_\star^k - \lambda_\star^{k+1}} &\leq h(\lambda_\star^k) - h(\lambda_\star^{k+1}) \nonumber \\
        \Leftrightarrow && \inner{s_\star^{k+1} - \tilde{s}}{\lambda_\star^k - \lambda_\star^{k+1} } &\leq h(\lambda_\star^k)  -  h(\lambda_\star^{k+1}) - \inner{\tilde{s}}{\lambda_\star^k - \lambda_\star^{k+1} } \nonumber  \\
        \Leftrightarrow &&  \inner{s_\star^{k+1} - \tilde{s}}{ \lambda_\star^k - \left( \lambda^k - \gamma_k s_\star^{k+1} \right) }  &\leq  h(\lambda_\star^k)  -  h(\lambda_\star^{k+1}) - \inner{\tilde{s}}{\lambda_\star^k - \lambda_\star^{k+1} }  \nonumber \\ 
        \Leftrightarrow && \gamma_k \inner{s_\star^{k+1} - \tilde{s}}{s_\star^{k+1}} &\leq \begin{multlined}[t]  h(\lambda_\star^k)  -  h(\lambda_\star^{k+1}) - \inner{\tilde{s}}{\lambda_\star^k - \lambda_\star^{k+1} } \\ + \inner{s_\star^{k+1}-\tilde{s}}{\lambda^k - \lambda_\star^k} \end{multlined} \label{eq:majorationA_intermediate}
    \end{align}
    which implies that 
    \begin{equation}
        \gamma_k \inner{s_\star^{k+1} - \tilde{s}}{s_\star^{k+1}} \leq \begin{multlined}[t]  h(\lambda_\star^k)- \inner{\tilde{s}}{\lambda_\star^k }  -  \left( h(\lambda_\star^{k+1}) - \inner{\tilde{s}}{\lambda_\star^{k+1} } \right) + 2M_s\varepsilon_{k} \end{multlined} \label{eq:majorationA}
    \end{equation}
    where in the last step we used \( \left| \inner{s_\star^{k+1} - \tilde{s}}{\lambda^k - \lambda_\star^k } \right| \leq  \Vert s_\star^{k+1} - \tilde{s} \Vert \Vert \lambda^k - \lambda_\star^k\Vert \leq (\Vert s_\star^{k+1} \Vert + \Vert \tilde{s} \Vert) \varepsilon_k \leq 2M_s \varepsilon_{k} \).
    Therefore, for \(l,N~\in~\N\) with \(l<N\): 
    \begin{equation*}
        \sum_{k=l}^{N-1} \gamma_k \inner{s_\star^{k+1} - \tilde{s}}{s_\star^{k+1}} \leq h(\lambda_\star^l) - \inner{\tilde{s}}{\lambda_\star^{l}} - \left( h(\lambda_\star^{N}) - \inner{\tilde{s}}{\lambda_\star^{N}} \right) + 2 M_s \sum_{k=l}^{N-1} \varepsilon_k .
    \end{equation*}
    By definition of \(h^*(\tilde{s})\), we also have that, for all \( N \in \N\), \linebreak\( - \left(h(\lambda_\star^{N}) - \inner{\tilde{s}}{\lambda_\star^{N}} \right)  \leq h^*(\tilde{s}) \).
    Thus, letting \(N \to \infty\) and using \ref{(A1)}:
    \begin{align}
        & \sum_{k=l}^\infty \gamma_k \inner{s_\star^{k+1} - \tilde{s}}{s_\star^{k+1}} \leq h(\lambda_\star^l) - \inner{\tilde{s}}{\lambda_\star^{l}} + h^*(\tilde{s}) + 2 M_s \sum_{k=l}^\infty \varepsilon_k  \label{eq:distance_to_asymptotic_plane_majoration} .
    \end{align}
    Observe that 
    \[
        \left\Vert s_\star^{k+1} \right\Vert^2 - \left\Vert \tilde{s} \right\Vert^2 = 2 \inner{s_\star^{k+1} - \tilde{s}}{s_\star^{k+1}} - \left\Vert s_\star^{k+1} - \tilde{s} \right\Vert^2 \leq  2 \inner{s_\star^{k+1} - \tilde{s} }{s_\star^{k+1}} , 
    \]
    therefore, by choosing \( l = 0\) in \eqref{eq:distance_to_asymptotic_plane_majoration},
    \begin{equation} \label{eq:majoration_sum_sk}
        \sum_{k=0}^\infty \gamma_k \left(\left\Vert s_\star^{k+1} \right\Vert^2 - \left\Vert \tilde{s} \right\Vert^2\right) \leq 2 \left( h(\lambda_\star^0) - \inner{\tilde{s}}{\lambda_\star^{0}} + h^*(\tilde{s}) \right) +  4 M_s \sum_{k=0}^\infty \varepsilon_k.
    \end{equation}
    Also, using \eqref{eq:majoration_s_star_k} (summed from index \(j=k+1\) to \(N-1\)), we have for \(k < N-1\),
    \( \left\Vert s_\star^{N} \right\Vert^2 \leq \left\Vert s_\star^{k+1} \right\Vert^2 + \sum_{j=k+1}^{N-1} \left( \tfrac{\varepsilon_{j}}{\gamma_{j}} \right)^2 \).
    Therefore,
    \begin{align*}
        \sum_{k=0}^{\changed{N-1}} \gamma_k \left\Vert s_\star^N \right\Vert^2 &\leq \sum_{k=0}^{\changed{N-1}} \gamma_k \left( \left\Vert s_\star^{\changed{k+1}} \right\Vert^2 + \sum_{i=\changed{k+1}}^{N-1} \left( \tfrac{\varepsilon_{i}}{\gamma_{i}} \right)^2 \right) \\
        &\leq \sum_{k=0}^{\changed{N-1}} \gamma_k \left\Vert s_\star^{\changed{k+1}} \right\Vert^2 + \sum_{k=0}^{\changed{N-1}} \gamma_k \sum_{i=\changed{k+1}}^{N-1} \left( \tfrac{\varepsilon_{i}}{\gamma_{i}} \right)^2 \\
        & \leq \sum_{k=0}^{\changed{N-1}} \gamma_k \left\Vert s_\star^{\changed{k+1}} \right\Vert^2 + \sum_{i=0}^{N-1} \left( \tfrac{\varepsilon_{i}}{\gamma_{i}} \right)^2 \sum_{k=0}^{\changed{i-1}} \gamma_k
    \end{align*}
    and using this inequality along with \eqref{eq:majoration_sum_sk} we have that 
    \begin{align*}
        \left( \sum_{k=0}^{\changed{N-1}} \gamma_k \right)\left( \left\Vert s_\star^N \right\Vert^2 - \left\Vert \tilde{s} \right\Vert^2 \right) & \added{ \leq   \sum_{k=0}^{N-1} \gamma_k \left( \left\Vert s_\star^{k+1} \right\Vert^2 - \left\Vert \tilde{s} \right\Vert^2 \right) + \sum_{i=0}^{N-1} \left( \tfrac{\varepsilon_{i}}{\gamma_{i}} \right)^2 \sum_{k=0}^{i-1} \gamma_k }\\
     &\leq \begin{multlined}[t] 2 \left( h(\lambda^0) - \inner{\tilde{s}}{\lambda^{0}} + h^*(\tilde{s}) \right) +  \changed{4} M_s \sum_{k=0}^{\changed{\infty}} \varepsilon_k \\  + \sum_{i=0}^{\changed{\infty}} \left( \tfrac{\varepsilon_{i}}{\gamma_{i}} \right)^2 \sum_{k=0}^{\changed{i-1}} \gamma_k ,  \end{multlined}
    \end{align*}
    therefore, by using \ref{(A3)} and \ref{(A1)} we have that 
    \( \left( \sum_{k=0}^{\changed{N-1}} \gamma_k \right) \left( \left\Vert s_\star^N \right\Vert^2 - \left\Vert \tilde{s} \right\Vert^2 \right) \) is bounded, thereby
    \begin{equation} \label{eq:rate_s_tilde}
        \left\Vert s_\star^N \right\Vert^2 - \left\Vert \tilde{s} \right\Vert^2  = O \left( \tfrac{1}{\sum_{k=0}^{\changed{N-1}} \gamma_k} \right).
    \end{equation}
    The rest of the proof is divided in two cases : if \(h^*(\sBar) < \infty \) or if \( h^*(\sBar) = \infty \).
    \medskip
    \item
    \noindent\textbullet \; \textbf{Case 1.}
    If \( h^*(\sBar) < \infty \), then we can choose \( \tilde s = \sBar \) in equation \eqref{eq:rate_s_tilde} which yields (a). 
    Having proven this rate for the norm of \(s_\star^k\), we now prove that the same rate holds for the norm of~\(s^k\). 
    We have, for \(N >0\), \( \left\Vert s_\star^N -  s^N  \right\Vert \leq \varepsilon_{N}/\gamma_{N-1} \) and \ref{(A2)} implies that \(  \varepsilon_{\changed{N}}/\gamma_{\changed{N-1}} =   O\left(\tfrac{1}{\sum_{k=0}^{\changed{N-1}} \gamma_k} \right) \).
     Therefore for \(N >0\)
    \begin{align*}
        \left\Vert s^N  \right\Vert^2 - \left\Vert \tilde{s} \right\Vert^2 & \leq \left\Vert s^N  - s_\star^N + s_\star^N   \right\Vert^2  - \left\Vert \tilde{s} \right\Vert^2 \\
        & \leq \left\Vert s^N - s_\star^N \right\Vert^2 + 2 \left\Vert s^N - s_\star^N \right\Vert \left\Vert s_\star^N  \right\Vert + \left\Vert s_\star^N  \right\Vert^2 - \left\Vert \tilde{s} \right\Vert^2   \\
        & \leq \left( \varepsilon_{N}/\gamma_{N-1} \right)^2 + 2 M_s \varepsilon_{N}/\gamma_{N-1} + \left\Vert s_\star^N  \right\Vert^2 - \left\Vert \tilde{s} \right\Vert^2   \\
        & \leq O\left( \tfrac{1}{\sum_{k=0}^{\changed{N-1}} \gamma_k} \right) ,
    \end{align*}
    which corresponds to (b). The two last convergence rates (c) and (d) are consequences of those just shown. 
    Since \(s_\star^N \in \cl(\range(\partial h)) \) and \(\tilde{s}\) is the projection of \(0\) on \(\cl(\changed{\range}(\partial h))\)  , we have \( \left\Vert s_\star^N -  \tilde{s} \right\Vert^2 = \left\Vert s_\star^N \right\Vert^2 - 2 \inner{s_\star^N}{\tilde{s}} + \left\Vert \tilde{s} \right\Vert^2 \leq \left\Vert s_\star^N \right\Vert^2 - \left\Vert \tilde{s} \right\Vert^2\).
    Therefore we also get \[ \left\Vert s_\star^N -  \tilde{s} \right\Vert = O\left(\tfrac{1}{\sqrt{\sum_{k=0}^{\changed{N-1}} \gamma_k}} \right) \] which corresponds to (c).
    We also have 
    \[ \left\Vert s^N -  \tilde{s} \right\Vert \leq \left\Vert s_\star^N -  \tilde{s} \right\Vert  + \left\Vert s^N - s_\star^N \right\Vert \leq \left\Vert s_\star^N -  \tilde{s} \right\Vert  + \tfrac{\varepsilon_{\changed{N}}}{\gamma_{\changed{N-1}}} = O\left(\tfrac{1}{\sqrt{\sum_{k=0}^{\changed{N-1}} \gamma_k}} \right) \] which is (d),
    where we used \ref{(A2)} to have \( \tfrac{\varepsilon_{\changed{N}}}{\gamma_{\changed{N-1}}}  = O\left(\tfrac{1}{\sum_{k=0}^{\changed{N-1}} \gamma_k} \right)\) which implies \(  \tfrac{\varepsilon_{\changed{N}}}{\gamma_{\changed{N-1}}} = O\left(\tfrac{1}{\sqrt{\sum_{k=0}^{\changed{N-1}} \gamma_k}} \right) \).
    \medskip
    \item
    \changed{\noindent\textbullet \; \textbf{Case 2.}
    if \( h^*(\sBar) = \infty \), then we can choose \(\tilde s \) as close to \(\sBar\) as desired in \eqref{eq:rate_s_tilde}.
    Indeed \(\Vert \sBar \Vert < M_s\) and \(\sBar\) is on the border of \(\mathrm{dom} h^*\) so we can always choose a \(\tilde s\) as close as desired to  \(\sBar\)  with \( \Vert \sBar \Vert < M_s\) and \( \sBar \in \mathrm{dom}( h^* )\).
    Since for any such \(\sBar\), \eqref{eq:rate_s_tilde} implies \( \left\Vert s_\star^N \right\Vert^2 - \left\Vert \tilde{s} \right\Vert^2  \to 0\) then we have 
     \( \left\Vert s_\star^N \right\Vert^2 - \left\Vert \sBar \right\Vert^2  \to 0\).
    Since \(\sBar\) is the element of smallest norm in \(\cl (\Scal)\) this implies \( s_\star^N \to \sBar \). And since \( \Vert s_\star^N - s^N \Vert \to 0\) we also obtain  \( s^N \to \sBar\).}

    \deleted{ \ref{(A3)} implies that \( \sum_{k=1}^\infty \left( \tfrac{\varepsilon_k}{\gamma_k} \right)^2 < \infty \), this fact along with \eqref{eq:majoration_s_star_k} shows that \(\left\Vert s_\star^k \right\Vert^2\) is a Cauchy sequence and that it therefore converges.
    \added{ \eqref{eq:majoration_s_star_k} implies that the nondecreasing sequence \( \sum_{0}^{k} \max(0, \left\Vert s_\star^{i+1} \right\Vert^2 - \left\Vert s_\star^i \right\Vert^2) \) is upper bounded by \(\sum_{k=1}^\infty \left( \varepsilon_k / \gamma_k \right)^2 < \infty \) (finite by \ref{(A3)})
     and therefore converges, likewise the nonincreasing sequence  \( \left\Vert s_\star^{k+1} \right\Vert^2 - \sum_{0}^{k+1}  \max(0, \left\Vert s_\star^{i+1} \right\Vert^2 - \left\Vert s_\star^i \right\Vert^2) \) is bounded below by \( - \sum_{k=1}^\infty \left( \varepsilon_k / \gamma_k \right)^2 > -\infty \) and therefore converges. 
    \(\left\Vert s_\star^{k+1} \right\Vert^2 \) is the sum of these two sequences therefore it converges.}
    We are going to show by contradiction that \changed{the limit of \(\left\Vert s_\star^{k+1} \right\Vert^2 \)} is \(\left\Vert \sBar \right\Vert^2\).
    Suppose that there exists \( \zeta > 0\), \(N\in \N\) such that \( \forall k \in \N^* , k \geq N \Rightarrow  \left\Vert s_\star^k \right\Vert \geq \left\Vert \overline s \right\Vert + \zeta \). By reindexing the \(s_\star^k\) and discarding the first terms, we can suppose without loss of generality that \( \forall k \in \N^*, \; \left\Vert s_\star^k \right\Vert \geq \left\Vert \overline s \right\Vert + \zeta \).
    We know that \( \sBar \in \cl(\mathrm{range } (\partial h)) = \cl(\mathrm{dom } (\partial h^*)) \); we can therefore choose \( (\tilde \lambda , \tilde s) \in \Y \times \mathrm{dom } (\partial h^*) \) such that \( \tilde \lambda \in \partial h^*(\tilde  s ) \) (which is equivalent to \(\tilde s \in \partial h ( \tilde \lambda)\)) and \( \left\Vert \tilde s \right\Vert \leq \left\Vert \overline s \right\Vert + \zeta/2 \).
     Naturally, \(  h^*(\tilde s ) < \infty \).
    Then, by constructing the convex, piecewise linear function 
    \begin{equation*}
        \tilde h: \lambda \mapsto  \max \left( \sup_{k\in \N^*} \inner{s_\star^k}{\lambda} - h^*(s_\star^k) \; ,\;  \inner{\tilde s}{\tilde \lambda} - h^*(\tilde s ) \right) ,
    \end{equation*} we \added{now explain} that \((\lambda^k)_{k\in \N} \) is an IPPA sequence on \( \tilde h \).
     \added{ Intuitively, \(\tilde h\) is the convex affine by parts function constructed from the tangent planes to \(h\) 
     at all the \(\lambda^k\), as well as the plane of slope \(\tilde s \) with highest intersect under the graph of \(h\) (this last plane is potentially tangent to \(h\) at some point, but could also ``touch'' \(h\) only at infinity). 
     Indeed \(\lambda \to \inner{s_\star^k}{\lambda} - h^*(s_\star^k) \) is the equation of the tangent plane to \(h\) at \(\lambda^k\) by definition of \(h^*\) and \(s_\star^k \in \partial h(\lambda^k)\).
    We see that, for all \(k \in \N\), \( \lambda_\star^{k+1} = \Prox_{\gamma_k \tilde h}(\lambda^k) \) because \( s_\star^{k+1} \in \partial \tilde h(\lambda^{k+1}) \) and \(0 = s_\star^{k+1} + \tfrac{\lambda_\star^{k+1} - \lambda^k}{\gamma_k} \) 
    and therefore, by verifying all the steps of the IPPA algorithm, we observe that \((\lambda^k)_{k\in \N} \) is an IPPA sequence on \( \tilde h \). }
    \added{Furthermore,} \(\tilde s\) is the smallest-norm element in \( \range (\partial \tilde h ) \) by construction.
    We therefore find ourselves in the first case treated in this proof, and we have shown that we should have \( s_\star^k \underset{ k\to \infty }{ \longrightarrow } \tilde s \), but this is a contradiction since \( \forall k \in \N, \; \left\Vert \tilde s \right\Vert \leq \left\Vert \overline s \right\Vert + \zeta/2 \leq \left\Vert s_\star^k \right\Vert - \zeta/2  \).
    \newline
    This contradiction shows that we must have \( \left\Vert s_\star^k \right\Vert \to \left\Vert \sBar \right\Vert \), meaning that \linebreak \( \left\Vert s_\star^k - \sBar \right\Vert^2 = \left\Vert s_\star^k \right\Vert^2 - 2 \inner{s_\star^k}{\overline s} + \left\Vert \sBar \right\Vert^2 \leq \left\Vert s_\star^k \right\Vert^2 - \left\Vert \sBar \right\Vert^2 \to 0 \).
    Since \( \left\Vert s_\star^k -  s^k  \right\Vert \leq \tfrac{\varepsilon_k}{\gamma_{\changed{k-1}}} \), and \ref{(A2)} implies that \(  \tfrac{\varepsilon_k}{\gamma_{\changed{k-1}}} \to 0  \), we also have that \( \left\Vert s^k -  \sBar \right\Vert \to 0 \) .}
    \end{proof}

In what follows, an important distinction in convergence rates is made between the case where \(h^*\) is subdifferentiable at \(\sBar\) and the case where it is not. The following proposition clarifies what exactly this condition means by giving equivalent formulations.
\begin{proposition} \cite[Section 2]{brondsted1965subdifferentiability} \label{prop:equivalent_formulations_g_Star_subdifferentiable}
Consider the proper closed convex function \(h\). The following properties are equivalent~:
    \begin{itemize}
        \item[(a)] there exists \( \overline \lambda \in \Y \) such that \( h^*(\sBar) + h( \overline \lambda ) = \inner{\sBar}{\overline{\lambda}} \)
        \item[(b)] there exists \( \overline \lambda \in \Y \) such that \( \overline{\lambda} \in \partial h^*(\sBar) \)  (i.e., \(h^* \) is subdifferentiable at \( \sBar \))
        \item[(c)] there exists \( \overline \lambda \in \Y \) such that \( \sBar \in \partial \changed{ h(\overline{\lambda} ) } \) (i.e., \(\sBar\) is in the range of \(\partial \changed{h} \))
    \end{itemize}
\end{proposition}

IPPA will tend to descend the curve of \(h\), but since \(h\) might be unbounded below, the iterates \(h(\lambda^k)\) might diverge to \(-\infty\), giving us little useful information.
 Another quantity that is natural to observe however is the vertical distance between the value iterates \(h(\lambda^k)\) to the asymptotical tangent plane of slope \(\sBar\) described by the equation \( \lambda \mapsto \inner{\lambda}{\sBar} - h^*(\sBar) \).
  This is the plane of smallest norm slope and highest intersect that is under the curve of \(h\). This vertical distance is described by the quantity \( h(\lambda^k) - \inner{\lambda^k}{\sBar} + h^*(\sBar) \).
   It turns out that this quantity converges to 0, with possibly a convergence rate, as shown by the following lemma.

\begin{lemma} \label{prop:convergence_vertical_dist_asymptotic_plane}
    Let \(h\) be a closed proper convex function and let the sequence  \((\lambda_\star^{k+1}, s_\star^{k+1}, \lambda^{k+1} , s^{k+1})_{k\in\N}\) be from an IPPA (\Cref{algo:IPPA})
     on \(h\) \changed{with errors \((\varepsilon_k)_{k\in\N^*}\) and step sizes \((\gamma_k)_{k\in\N}\) satisfying \Cref{assumption:errors_and_step_sizes}}.
     \newline Let \(\sBar = \argmin_{s \in \cl(\range(\partial h) )} \|s\|^2\). 
If \(h^*(\sBar) < \infty\), we have \( h(\lambda_\star^k) + h^*(\sBar) - \inner{\lambda_\star^k}{\sBar} \to 0 \).
If furthermore~\(\sBar\) is in the range of \(\partial h\), or equivalently \(h^*\) is subdifferentiable at~\(\sBar\),
 then \[ h(\lambda_\star^k) + h^*(\sBar) - \inner{\lambda_\star^k}{\sBar} = O\left(\tfrac{1}{\sqrt{\sum_{i=\changed{0}}^{\changed{k-1}} \gamma_i}} \right).\footnote{ \added{The interested reader may notice that this result can also be stated in terms of the $\varepsilon$-subdifferential, but this point of view is not useful for our purposes.}} \]
\end{lemma}

\begin{proof}
We first establish that  \(h(\lambda_\star^{k+1}) - \inner{\lambda_\star^{k+1}}{\sBar} \leq h(\lambda_\star^k) - \inner{\lambda_\star^k}{\sBar} + M_s \varepsilon_k\).
By convexity of \(h\) and \( s_\star^{k+1} \in \partial h(\lambda_\star^{k+1}) \) we have
\begin{align}
    h(\lambda_\star^{k+1}) & \leq h(\lambda_\star^k) - \inner{s_\star^{k+1}}{\lambda_\star^k - \lambda_\star^{k+1}} \nonumber \\ 
    & = h(\lambda_\star^k) - \inner{s_\star^{k+1}}{\lambda^k - \lambda_\star^{k+1}} + \inner{s_\star^{k+1}}{\lambda^k - \lambda_\star^{k}} \nonumber \\
    & \leq h(\lambda_\star^k) - \gamma_k \left\Vert s_\star^{k+1} \right\Vert^2 + M_s \varepsilon_k \nonumber
\end{align}
which leads to
\begin{align}
    h(\lambda_\star^{k+1}) - \inner{\lambda_\star^{k+1}}{\sBar} &\leq h(\lambda_\star^k) - \gamma_k \left\Vert s_\star^{k+1} \right\Vert^2  -  \inner{\lambda_\star^{k+1}}{\sBar}  + M_s \varepsilon_k \nonumber \\
        &\leq h(\lambda_\star^k)  - \gamma_k \left\Vert s_\star^{k+1} \right\Vert^2 -  \inner{\lambda^k - \gamma_k s_\star^{k+1}}{\sBar}  + M_s \varepsilon_k \nonumber \\
        &\leq \begin{multlined}[t] h(\lambda_\star^k)  - \gamma_k \left\Vert s_\star^{k+1} \right\Vert^2 + \gamma_k \inner{s_\star^{k+1}}{\sBar} - \inner{\lambda_\star^k}{\sBar}  \\ + \Vert \sBar \Vert \left\Vert \lambda_\star^k - \lambda^k \right\Vert + M_s \varepsilon_k \end{multlined} \nonumber   \\
        &\leq h(\lambda_\star^k) - \inner{\lambda_\star^k}{\sBar} + 2 M_s \varepsilon_k \label{eq:distance_to_asymptotic_plane_decrease}
\end{align}
where we used the inequality  \changed{\( \inner{s_\star^{k+1}}{\sBar} \leq \Vert s_\star^{k+1} \Vert \Vert \sBar \Vert \leq \left\Vert s_\star^{k+1} \right\Vert^2 \) (since \(\sBar = \Proj_{\overline{\Scal}}(0)\) and \(s_\star^{k+1} \in \Scal\))}. 
\eqref{eq:distance_to_asymptotic_plane_decrease} can be understood as meaning that if there was no error performed, the sequence \((h(\lambda_\star^k) - \inner{\lambda_\star^k}{\sBar})_{k\in\N}\) would be nonincreasing.
 The rest of the proof consists of the analysis of two separate cases. 
\bigskip
\item
\noindent\textbullet \; \textbf{Case 1 \( \sBar \notin \range(\partial h) \).}

Pick any \(\delta > 0\) and choose a \(\mu^0  \in \R^m \) such that \( h(\mu^0) + h^*(\sBar) - \inner{\sBar}{\mu^0} < \delta\), where the existence of \(\mu^0 \) stems from the fact that \(\lambda \mapsto \inner{\lambda}{\sBar} - h(\lambda) \) is upper semi-continuous and the definition \( h^*(\sBar) = \sup_{\lambda \in \Y} \left( \inner{\lambda}{\sBar} - h(\lambda) \right) \).

Now we define the sequence \((\mu^k)_{k\in\N}\) by \(\forall k \in \N, \; \mu^{k+1} = \Prox_{\gamma_k h}(\mu^k)\).
The (exact) proximal point operator is nonexpansive, hence \( \left\Vert \mu^{k+1} - \lambda^{k+1} \right\Vert \leq \left\Vert \mu^{k+1} - \lambda_\star^{k+1} \right\Vert + \varepsilon_{k+1} \leq \left\Vert \mu^k - \lambda^k \right\Vert  + \varepsilon_{k+1} \). 
By summing all these inequalities for \(i\) from 0 to \(N\) we have 
\begin{equation} \label{eq:ineq_distance_lambdk_muk}
    \forall k \in \N \;:\; \left\Vert \mu^k - \lambda_\star^k \right\Vert \leq \left\Vert \mu^0 - \lambda^0 \right\Vert + \sum_{i=1}^k \varepsilon_i  .
\end{equation}
Further, convexity of \(h\) and \( s_\star^k \in \partial h(\lambda_\star^k) \) allows reaching
\begin{align}
    h(\lambda_\star^k) - h(\mu^k) - \inner{\sBar}{\lambda_\star^k - \mu^k} &\leq \inner{s_\star^k}{\lambda_\star^k - \mu^k} - \inner{\sBar}{\lambda_\star^k - \mu^k} \nonumber \\
    &\leq \inner{s_\star^k - \sBar}{\lambda_\star^k - \mu^k  } \nonumber \\
    &\leq \left\Vert s_\star^k - \sBar \right\Vert \left( \left\Vert \lambda^0 - \mu^0 \right\Vert + \sum_{i=1}^\infty \varepsilon_i \right) , \label{eq:diff_between_distance_to_asymp_hyperplane_mu_lambd}
\end{align}
where we used \eqref{eq:ineq_distance_lambdk_muk} in the last inequality.
Furthermore \( h(\mu^k) - \inner{\mu^k}{\sBar}\) is nonincreasing (as indicated by \eqref{eq:distance_to_asymptotic_plane_decrease} when \(\varepsilon_k\) is substituted by 0 and \(\lambda\) by \(\mu\)), hence
\begin{equation} \label{eq:majoration_distance_asymp_hyperplane_mu}
    \forall k \in \N \;: \; h(\mu^k) + h^*(\sBar) - \inner{\mu^k}{\sBar} \leq \delta .
\end{equation}
Combining \eqref{eq:diff_between_distance_to_asymp_hyperplane_mu_lambd} and \eqref{eq:majoration_distance_asymp_hyperplane_mu} we get
\begin{align}
    h(\lambda_\star^k) + h^*(\sBar) - \inner{\sBar}{\lambda_\star^k} &\leq \begin{multlined}[t] h(\mu^k) + h^*(\sBar) - \inner{\mu^k}{\sBar} \\ + \left\Vert s_\star^k - \sBar \right\Vert \left( \left\Vert \lambda^0 - \mu^0 \right\Vert + \sum_{i=1}^\infty \varepsilon_i \right) \label{eq:dominance of the distance to the asymptotic hyperplane} \end{multlined} \\
    &\leq \delta + \underbrace{\left\Vert s_\star^k - \sBar \right\Vert }_{\to 0} \left( \left\Vert \lambda^0 - \mu^0 \right\Vert + \sum_{i=0}^\infty \varepsilon_i \right). \nonumber 
\end{align}
\added{Here, \(\left\Vert s_\star^k - \sBar \right\Vert \to 0\) follows from \Cref{prop:sk_convergence}}.
Therefore, since \(\sum_{i=0}^\infty \varepsilon_i < \infty\) by \ref{(A1)}, we have \( \limsup_{k \to \infty} h(\lambda_\star^k) + h^*(\sBar) - \inner{\sBar}{\lambda_\star^k} \leq \delta \). 
We also have, from the definition of \(h^*\), that \( h(\lambda_\star^k) + h^*(\sBar) - \inner{\sBar}{\lambda_\star^k} \geq 0\).
The choice of \(\delta\) being arbitrary, we have shown that \( h(\lambda_\star^k) + h^*(\sBar) - \inner{\sBar}{\lambda_\star^k} \to 0 \).

\bigskip
\item
\noindent\textbullet \; \textbf{Case 2 \( \sBar \in \range(\partial h) \).}

In the case where \( \sBar \in \mathrm{dom}(\partial h^*)\), the reasoning is similar to that of Case~1, except that we can directly choose \(\mu^0 \in \partial h^*(\sBar) \). 
Since \( h(\mu^k) - \inner{\mu^k}{\sBar}\) is nonincreasing (as indicated by \eqref{eq:distance_to_asymptotic_plane_decrease} when \(\varepsilon_k\) is substituted by 0 and \(\lambda\) by~\(\mu\)), we have \(\forall k \in \N , \;  -h^*(\sBar) = \inf_{\mu \in \Y} h(\mu) - \inner{\mu}{\sBar} \leq h(\mu^k) - \inner{\mu^k}{\sBar} \leq h(\mu^0) - \inner{\mu^0}{\sBar} = -h^*(\sBar)\), so \( h(\mu^k) - \inner{\mu^k}{\sBar} = -h^*(\sBar)\).
 Using this equality and \eqref{eq:dominance of the distance to the asymptotic hyperplane} leads to
\begin{align*}
    & h(\lambda_\star^k) + h^*(\sBar) - \inner{\sBar}{\lambda_\star^k} \leq  \left\Vert s_\star^k - \sBar \right\Vert \left( \left\Vert \lambda^0 - \mu^0 \right\Vert + \sum_{i=1}^\infty \varepsilon_i \right) = O \left( \tfrac{1}{ \sqrt{ \sum_{i=0}^{\changed{k-1}} \gamma_i}} \right)
\end{align*}
where we used \(\sum_{i=0}^\infty \varepsilon_i < \infty\) from \ref{(A1)} and the rate from \Cref{prop:sk_convergence} .
\end{proof}

We now study the convergence of the convex conjugate iterates \( (h^*(s_\star^k))_{k\in\N}\).

\begin{theorem} \label{thm:convergence_h_s_k_to_h_s_bar}
    Let \(h\) be a closed, proper and convex function. Let the sequence \((\lambda_\star^{k+1}, s_\star^{k+1}, \lambda^{k+1} , s^{k+1})_{k\in\N}\) be from an IPPA (\Cref{algo:IPPA}) on \(h\)
    \changed{with errors \((\varepsilon_k)_{k\in\N^*}\) and step sizes \((\gamma_k)_{k\in\N}\) satisfying \Cref{assumption:errors_and_step_sizes}}.
     \newline Let \(\sBar = \argmin_{s \in \cl(\range(\partial h) )} \|s\|^2\). 
    
    The sequence \( h^*(s_\star^{k})_{k\in \N^*} \) converges in \(\R \cup \{+\infty\}\) to \( h^*(\sBar)\).
\end{theorem}

\begin{proof}
If \(h^*(\sBar) = \infty\), then the result is clear from the fact that \(h^*\) is lower-semicontinuous and \( s_\star^k \to \sBar\).
In the rest of this proof, we therefore assume that \(h^*( \sBar) < \infty\). The proof consists of showing that \( \limsup_{K \in \N} h^*(s^K) - h^*(\sBar) \leq 0  \), since \(h^*\) is lower-semicontinuous that will be sufficient to obtain the result \(h^*(s^K) \to h^*(\sBar)\). 
To show \(\limsup_{K \in \N} h^*(s^K) - h^*(\sBar) \leq 0\) we will majorize \(h^*(s^K) - h^*(\sBar)\) by various vanishing quantities including the distance to the asymptotic plane of \Cref{prop:convergence_vertical_dist_asymptotic_plane}.
\deleted{The derivation of this proof was facilitated by the performance estimation approach (PEP) \cite{drori2014performance,taylor2017exact}. \added{Note that since the following inequalities were derived using PEP, they may appear intricate upon first reading.}}

For any \(k \in \N, K \in \N^*\) with \(k\leq K\),
\begin{align*}
    &   \inner{s_\star^k}{s_\star^K - \sBar} +   \sum_{j=k}^{K-1} \inner{s_\star^{j+1}}{s_\star^{j+1} - s_\star^j } - 2   \inner{\sBar}{s_\star^K - \sBar} + \inner{s_\star^k}{s_\star^k - \sBar} \\
    & \quad \added{= \begin{multlined}[t] \sum_{j=k}^{K-1} \left( \left\Vert s_\star^{j+1} \right\Vert^2 - \inner{s_\star^{j+1}}{s_\star^j}  \right) + \tfrac{\left\Vert s_\star^k \right\Vert^2}{2} - \tfrac{\left\Vert s_\star^K \right\Vert^2}{2} \\ + 2 \left( \Vert\bar{s}\Vert^2 + \tfrac{\Vert s_\star^K\Vert^2}{4} + \tfrac{\Vert s_\star^k\Vert^2}{4} - \inner{\bar{s}}{s_\star^K} - \inner{\bar{s}}{s_\star^k} + \tfrac{1}{2}\inner{s_\star^k}{s_\star^K} \right) \end{multlined}}\\
    & \quad \changed{=} \sum_{j=k}^{K-1} \left( \left\Vert s_\star^{j+1} \right\Vert^2 - \inner{s_\star^{j+1}}{s_\star^j}  \right) + \tfrac{\left\Vert s_\star^k \right\Vert^2}{2} - \tfrac{\left\Vert s_\star^K \right\Vert^2}{2} + 2 \left\Vert \sBar - \tfrac{s_\star^K + s_\star^k}{2} \right\Vert^2 \\
    & \quad =  \begin{multlined}[t] \sum_{j=k}^{K-1} \left( \tfrac{1}{2}\left\Vert s_\star^{j+1} \right\Vert^2 - \inner{s_\star^{j+1}}{s_\star^j}  \right) +  \sum_{j=k+1}^{K}\left( \tfrac{1}{2}\left\Vert s_\star^{j} \right\Vert^2 \right) + \tfrac{\left\Vert s_\star^k \right\Vert^2}{2} - \tfrac{\left\Vert s_\star^K \right\Vert^2}{2} \\ + 2 \left\Vert \sBar - \tfrac{s_\star^K + s_\star^k}{2} \right\Vert^2 \end{multlined}  \\
    & \quad = \sum_{j=k}^{K-1} \left( \tfrac{1}{2}\left\Vert s_\star^{j+1} -  s_\star^j  \right\Vert^2  \right) + 2 \left\Vert \sBar - \tfrac{s_\star^K + s_\star^k}{2} \right\Vert^2  \\
    & \quad \geq 0 ,
\end{align*}
and since \(\inner{\sBar}{s_\star^K - \sBar} \geq 0\) \added{(by definition \(\bar{s} = \Proj_{\Scal}(0)\))} we obtain
\begin{equation} \label{eq:ineq_PEP1}
    -\inner{s_\star^k}{s_\star^K - \sBar} \leq \inner{s_\star^k}{s_\star^k - \sBar} +   \sum_{j=k}^{K-1} \inner{s_\star^{j+1}}{s_\star^{j+1} - s_\star^j } .
\end{equation}

But we also observe that for all \(j\), using the convexity of \(h\) and \(s_\star^{j+1} \in \partial h (\lambda_\star^{j+1})\) and \(s_\star^{j} \in \partial h (\lambda_\star^{j})\), 
{\allowdisplaybreaks
\begin{align*}
    \inner{s_\star^{j+1}}{s_\star^{j+1} - s_\star^j } & =  \tfrac{1}{\gamma_j}\inner{ \lambda^j - \lambda_\star^{j+1}}{s_\star^{j+1} - s_\star^j } \\
    & =  \tfrac{1}{\gamma_j}\inner{\lambda_\star^j - \lambda_\star^{j+1}}{s_\star^{j+1} - s_\star^j } + \tfrac{1}{\gamma_j}\inner{\lambda^{j} - \lambda_\star^j}{s_\star^{j+1} - s_\star^j } \\
    & \leq  0 + \tfrac{1}{\gamma_j}\Vert\lambda^{j} - \lambda_\star^j\Vert \left( \Vert s_\star^{j+1}\Vert +  \Vert s_\star^j \Vert \right) \\
    & \leq  2M_s\tfrac{\varepsilon_{j}}{\gamma_j}.
\end{align*}}
Therefore \eqref{eq:ineq_PEP1} becomes 
\begin{equation} \label{eq:ineq_PEP1b}
    -\inner{s_\star^k}{s_\star^K - \sBar} \leq \inner{s_\star^k}{s_\star^k - \sBar} +  2M_s  \sum_{j=k}^{K-1} \tfrac{\varepsilon_{j}}{\gamma_j} .
\end{equation}
Further, the convexity of \(h^* \) and the expansion \(\lambda_\star^K - \lambda^l = - \sum_{i=l+1}^{K-1} \gamma_{i-1} s^i - \gamma_{\changed{K-1}}  s_\star^K \) allows writing
{\allowdisplaybreaks
\begin{align*}
    & h^*(s_\star^K) - h^*(\sBar) \leq \inner{\lambda_\star^K}{s_\star^K - \sBar} \\
    & \changed{=} \inner{\lambda^l}{s_\star^K - \sBar}  + \inner{\lambda_\star^K - \lambda^l}{s_\star^K - \sBar}    \\
    &  \changed{=} \begin{multlined}[t] \inner{\lambda^l}{s_\star^K - \sBar} - \sum_{i=l+1}^{K-1} \gamma_{i-1} \inner{s^i}{s_\star^K - \sBar} - \gamma_{K-1} \inner{s_\star^K}{s_\star^K - \sBar} \\ \text{ \deleted{(using $\inner{s^i}{s_\star^K - \sBar} \geq 0$)}} \end{multlined} \\
    &  \changed{=} \begin{multlined}[t] \inner{\lambda^l}{s_\star^K - \sBar} - \sum_{i=l+1}^{K-1} \gamma_{i-1} \inner{s_\star^i}{s_\star^K - \sBar}  - \sum_{i=l+1}^{K-1} \gamma_{i-1} \inner{s^i - s_\star^i}{s_\star^K - \sBar} \\ - \gamma_{K-1} \inner{s_\star^K}{s_\star^K - \sBar} \deleted{\text{ (using $\inner{s^i - s_\star^i}{s_\star^K - \sBar} \geq 0$)}} \end{multlined} \\
    &  \leq \begin{multlined}[t] \inner{\lambda^l}{s_\star^K - \sBar}  - \sum_{i=l+1}^{K-1} \gamma_{i-1} \inner{s_\star^i}{s_\star^K - \sBar} + 2M_s \sum_{i=l+1}^{K-1} \varepsilon_i - \gamma_{K-1} \inner{s_\star^K}{s_\star^K - \sBar}  \\ \changed{\left(\text{using } \inner{s^i - s_\star^i}{s_\star^K - \sBar} \leq (\Vert s_\star^K \Vert + \Vert \sBar \Vert )\Vert s^i - s_\star^i \Vert \leq 2 M_s \tfrac{\varepsilon_i}{\gamma_{i-1}} \right) } \end{multlined} \\
    & \leq \begin{multlined}[t] \inner{\lambda^l}{s_\star^K - \sBar}  + \sum_{i=l+1}^{K} \gamma_{i-1} \left( \inner{s_\star^i}{s_\star^i - \sBar} +  2M_s \sum_{j=i}^{K-1} \tfrac{\varepsilon_{j}}{\gamma_j} \right) \\ + 2M_s \sum_{i=l+1}^{K-1} \varepsilon_i \text{ (using \eqref{eq:ineq_PEP1b})} \end{multlined} \\
    & \leq \deleted{ \begin{multlined}[t] \inner{\lambda^l}{s_\star^K - \sBar}  + \sum_{i=l+1}^{\infty} \gamma_{i-1}  \inner{s_\star^i}{s_\star^i - \sBar} +  2M_s \sum_{j=l+1}^{\infty} \sum_{i=l+1}^{j}  \gamma_{i-1} \tfrac{\varepsilon_{j}}{\gamma_j}  \\ + 2M_s \sum_{i=l+1}^{K-1} \varepsilon_i \quad (\text{using} \inner{s_\star^i}{s_\star^i - \sBar} \geq 0) \end{multlined} } \\
    & \leq \begin{multlined}[t] \inner{\lambda^l}{s_\star^K - \sBar}  + h(\lambda_\star^{l+1}) - \inner{\sBar}{\lambda_\star^{l+1}} + h^*(\sBar)  +  2M_s \sum_{j=l+1}^{\infty} \sum_{i=l+1}^{j}  \gamma_{i-1} \tfrac{\varepsilon_{j}}{\gamma_j} \\ + 4 M_s \sum_{i=l+1}^{\infty} \varepsilon_i  \text{ (using \eqref{eq:distance_to_asymptotic_plane_majoration})}. \end{multlined}
\end{align*}}
We now show that for \(l\) and \(K\) sufficiently large the term on the right of that inequality can be made as small as desired. 

For any \( \delta >0\), using \Cref{prop:convergence_vertical_dist_asymptotic_plane} we can choose \( l\) sufficiently large such that
\begin{equation*}
    h(\lambda_\star^{l+1}) - \inner{\sBar}{\lambda_\star^{l+1}} + h^*(\sBar) \leq \tfrac{\delta}{3} .
\end{equation*}
Due to \ref{(A4)} and \ref{(A1)}, one can also choose \(l\) larger and sufficiently large such that 
\[  2M_s \sum_{j=l+1}^{\infty}  \sum_{i=l+1}^{j}  \gamma_{i-1}\tfrac{\varepsilon_{j}}{\gamma_j}  + \changed{4} M_s \sum_{i=l+1}^{\infty} \varepsilon_i \leq \tfrac{\delta}{3} . \]
Due to \Cref{prop:sk_convergence} we also have, for a choice of \( K \) larger than \(l\) and sufficiently large, that \( \left\Vert \lambda^{\changed{l}} \right\Vert  \left\Vert  s^K - \sBar \right\Vert  \leq \tfrac{\delta}{3} \). Therefore, for \(K\) sufficiently large, \(h^*(s^K) - h^*(\sBar) \leq \delta \).
We have shown that \( \limsup_{K \in \N} h^*(s^K) - h^*(\sBar) \leq 0  \). Since \( h^* \) is lower-semicontinuous, we have  \(  h^*(s^K) - h^*(\sBar) \to 0  \)
\end{proof}

We can obtain a convergence rate under some subdifferentiability condition as shown in \Cref{thm:convergence rate h^*(s^k) - h^*(sBAR)}.

\begin{theorem} \label{thm:convergence rate h^*(s^k) - h^*(sBAR)}
    Let \(h\) be a closed proper convex function. Let the sequence  \((\lambda_\star^{k+1}, s_\star^{k+1}, \lambda^{k+1} , s^{k+1})_{k\in\N}\) be from an IPPA (\Cref{algo:IPPA}) on \(h\)
    \changed{with errors \((\varepsilon_k)_{k\in\N^*}\) and step sizes \((\gamma_k)_{k\in\N}\) satisfying \Cref{assumption:errors_and_step_sizes}}.
     \newline Let \(\sBar = \argmin_{s \in \cl(\range(\partial h) )} \|s\|^2\). 

If \(h^*\) is subdifferentiable at \(\sBar\), or equivalently if there exists \( \overline\lambda \) such that \( \sBar \in \partial h(\overline{\lambda}) \), then we have the following convergence rate:
\begin{equation*}
    \left| h^*(s_\star^k) - h^*(\sBar) \right| = O\left(\tfrac{1}{\sqrt{\sum_{i=0}^{\changed{k-1}} \gamma_i}} \right) .
\end{equation*}
\end{theorem}
 
\begin{proof}
The proof consists of finding an upper and a lower bound of \(h^*(s_\star^k) - h^*(\sBar)\) that are proportional to \(\left\Vert s_\star^k - \sBar \right\Vert\), then the desired convergence rate of \(h^*(s_\star^k) - h^*(\sBar)\) will stem from the convergence rate of \(\left\Vert s_\star^k - \sBar \right\Vert\).
We first show that \( \lambda_\star^k + \sum_{i=1}^k \gamma_i \sBar  - \overline{\lambda } \) is bounded. In the following, we will use \( \inner{\sBar}{s_\star^{k+1} - \sBar} \geq 0 \) and \(\inner{\lambda_\star^{k+1}  - \overline{\lambda }}{s_\star^{k+1} - \sBar} \geq 0\) (by convexity of \(h\) and the fact that \(\sBar \in \partial h(\overline{\lambda}) \), \(s_\star^{k+1} \in \partial h(\lambda_\star^{k+1}) \)).

\added{We first show that \( \lambda_\star^{k} \) can be approximated with finite error by \( \overline{\lambda } -\sum_{i=0}^{k-1} \gamma_i \sBar \). }
\begingroup
\allowdisplaybreaks[4]
\begin{align*} 
& \left\Vert \lambda_\star^{k+1} + \sum_{i=\changed{0}}^{\changed{k}} \gamma_i \sBar - \overline{\lambda } \right\Vert^2 = \\ 
& \qquad \inner{\lambda_\star^{k+1} +\sum_{i=\changed{0}}^{\changed{k}} \gamma_i \sBar - \overline{\lambda }}{\lambda_\star^{k} + (\lambda^{k} - \lambda_\star^{k})  - \gamma_k (s_\star^{k+1} - \sBar) + \sum_{i=\changed{0}}^{\changed{k-1}} \gamma_i \sBar - \overline{\lambda } }  \\
& \quad  \leq \begin{multlined}[t] \inner{\lambda_\star^{k+1} +\sum_{i=\changed{0}}^{\changed{k}} \gamma_i \sBar - \overline{\lambda }}{\lambda_\star^{k} - \gamma_k ( s_\star^{k+1} - \sBar) + \sum_{i=\changed{0}}^{\changed{k-1}} \gamma_i \sBar - \overline{\lambda } } \\ + \varepsilon_k \left\Vert \lambda_\star^{k+1} + \sum_{i=\changed{0}}^{\changed{k}} \gamma_i \sBar - \overline{\lambda } \right\Vert \end{multlined} \\
& \deleted{ \begin{multlined} \quad \changed{=} \inner{\lambda_\star^{k+1} + \sum_{i=\changed{0}}^{\changed{k}} \gamma_i \sBar - \overline{\lambda }}{\lambda_\star^{k}  + \sum_{i=\changed{0}}^{\changed{k-1}} \gamma_i \sBar - \overline{\lambda } } \qquad \qquad \\ - \gamma_k \inner{\lambda_\star^{k+1} + \sum_{i=\changed{0}}^{\changed{k}} \gamma_i \sBar - \overline{\lambda }}{s_{\changed{\star}}^{k+1} - \sBar}  \\  + \varepsilon_k \left\Vert \lambda_\star^{k+1} + \sum_{i=\changed{0}}^{\changed{k}} \gamma_i \sBar - \overline{\lambda } \right\Vert  \end{multlined}} \\
& \quad  \begin{multlined} \changed{=} \inner{\lambda_\star^{k+1} + \sum_{i=\changed{0}}^{\changed{k}} \gamma_i \sBar - \overline{\lambda }}{\lambda_\star^{k}  + \sum_{i=\changed{0}}^{\changed{k-1}} \gamma_i \sBar - \overline{\lambda }} - \gamma_k \inner{\lambda_\star^{k+1}  - \overline{\lambda }}{s_{\changed{\star}}^{k+1} - \sBar} \\ - \gamma_k \sum_{i=\changed{0}}^{\changed{k}} \gamma_i \inner{\sBar}{s_\star^{k+1} - \sBar}  + \varepsilon_k \left\Vert \lambda_\star^{k+1} + \sum_{i=\changed{0}}^{\changed{k}} \gamma_i \sBar - \overline{\lambda } \right\Vert \end{multlined} \\
& \quad \leq  \inner{\lambda_\star^{k+1} + \sum_{i=\changed{0}}^{\changed{k}} \gamma_i \sBar - \overline{\lambda }}{\lambda_\star^{k}  + \sum_{i=\changed{0}}^{\changed{k-1}} \gamma_i \sBar - \overline{\lambda }}  + 0 + \varepsilon_k \left\Vert \lambda_\star^{k+1} + \sum_{i=\changed{0}}^{\changed{k}} \gamma_i \sBar - \overline{\lambda } \right\Vert  \\
& \quad \leq \left\Vert \lambda_\star^{k+1} + \sum_{i=\changed{0}}^{\changed{k}} \gamma_i \sBar - \overline{\lambda } \right\Vert \left\Vert \lambda_\star^{k}  +\sum_{i=\changed{0}}^{\changed{k-1}} \gamma_i \sBar - \overline{\lambda } \right\Vert  + \varepsilon_k \left\Vert \lambda_\star^{k+1} + \sum_{i=\changed{0}}^{\changed{k}} \gamma_i \sBar - \overline{\lambda } \right\Vert ,
\end{align*}
therefore \( \left\Vert \lambda_\star^{k+1} + \sum_{i=\changed{0}}^{\changed{k}} \gamma_i \sBar - \overline{\lambda } \right\Vert \leq \left\Vert \lambda_\star^{k}  + \sum_{i=\changed{0}}^{\changed{k-1}} \gamma_i \sBar - \overline{\lambda } \right\Vert + \varepsilon_k \). This implies that \( \left\Vert \lambda_\star^{k} + \sum_{i=\changed{0}}^{\changed{k-1}} \gamma_i \sBar - \overline{\lambda } \right\Vert \) is bounded by \( \left\Vert  \lambda^{0} - \overline{\lambda } \right\Vert + \sum_{i=1}^\infty \varepsilon_i\).
We use this bound to establish the following inequality:
\begin{align*}
    h^*(s_\star^k) - h^*(\sBar) &\leq \inner{\lambda_\star^k}{s_\star^k - \sBar} \\
     &=  \begin{multlined}[t] \inner{\lambda_\star^k +  \sum_{i=0}^{\changed{k-1}} \gamma_i \sBar - \overline{\lambda}}{s_\star^k - \sBar} + \inner{\overline{\lambda}}{s_\star^k - \sBar} \\ - \sum_{i=0}^{\changed{k-1}} \gamma_i \inner{\sBar}{s_\star^k - \sBar} \end{multlined} \\
    &\leq \begin{multlined}[t] \left( \left\Vert  \lambda_\star^{k} + \sum_{i=0}^{\changed{k-1}} \gamma_i \sBar - \overline{\lambda } \right\Vert +  \left\Vert \overline{\lambda} \right\Vert \right) \left\Vert s_\star^k - \sBar \right\Vert \\  \left( \text{using } \inner{\sBar}{s_\star^k - \sBar} \geq 0 \text{ \added{because $\bar{s} = \Proj_{\Scal}(0)$}}\right)  \end{multlined} \\
    &\leq  \left(  \left\Vert  \lambda^{0} - \overline{\lambda } \right\Vert + \sum_{i=1}^\infty \varepsilon_i  + \left\Vert \overline{\lambda} \right\Vert \right) \left\Vert s_\star^k - \sBar \right\Vert.
\end{align*}
\endgroup
This provides an upper bound on \(h^*(s_\star^k) - h^*(\sBar)\). By convexity, one also has \(h^*(s_\star^k) - h^*(\sBar) \geq \inner{\overline{\lambda}}{s_\star^k - \sBar} \geq - \left\Vert \overline{\lambda} \right\Vert \left\Vert s_\star^k - \sBar \right\Vert  \) 
which provides a lower bound. 
Since \( \left\Vert s^k - \sBar \right\Vert = O\left(\tfrac{1}{\sqrt{\sum_{i=\changed{0}}^{\changed{k-1}} \gamma_i}} \right) \) by \Cref{prop:sk_convergence} we have \( h^*(s_\star^k) - h^*(\sBar) = O\left(\tfrac{1}{\sqrt{\sum_{i=\changed{0}}^{\changed{k-1}} \gamma_i}} \right) \).
\end{proof}

The next section applies the results obtained in this section for IPPA to the particular case of IALM where we choose \(h = \nu^* \).

%% file: sec_convergence_of_IALM.tex
\section{Convergence of IALM} \label{sec:convergence_ialm}

\input{diagram_part3.tex}

This section establishes the convergence properties of IALM in the case where the convex optimization problem might be infeasible.
 The relationship between the IALM on~\eqref{eq:convex_min_problem} and the IPPA on the corresponding dual function established in \Cref{sec:IALM is IPPA} allows us to translate the results of \Cref{sec:IPPA} into results on IALM.

Of particular interest to us is the case where the \changed{ limiting behavior of IALM can be interpreted in terms of the closest feasible problem \eqref{eq:bilevel optimization problem}}.
\added{The value of the shifted problem with the minimal constraint transgression is, by definition,} \(\nu(\sBar)\), and it can be written with simple algebraic manipulation from \eqref{eq:convex_min_problem_shifted} as the bilevel optimization problem of the least constraint transgression:
\begin{equation} \label{eq:bilevel optimization problem} 
\begin{array}{rl}
    \displaystyle \nu(\sBar) = \min_{x\in \X} & f(x) \\
    \text{s.t.} & x \in \argmin\limits_{ \tilde x \in  \mathrm{dom}(f) } \left\Vert C(\tilde x) - \Proj_{\KK} (C(\tilde x)) \right\Vert  .
\end{array}
\end{equation}
\added{Whenever \eqref{eq:bilevel optimization problem} is defined, we also call it the closest feasible problem.
 We note \(\mathrm{Sol}_{\sBar} = \{ x \in \X \mid f(x) = \nu(\sBar) \text{ and } C(x) \in \KK + \sBar \}\) the solution set of the closest feasible problem.}

We define three types of convergence to the closest feasible problem, which will be established by the theorems in this section.

{
\begin{definition}[Convergence and quantitative convergence of IALM] \label{def:convergence_closest_feasible}
Let \( (x^k, y^k, \lambda^k, s^k )_{k \in \N} \) be a sequence generated by an IALM (\Cref{algo:IALM}) associated with problem \eqref{eq:convex_min_problem}. We say that IALM \textbf{simply converges} to the closest feasible problem when 
\begin{itemize}
    \item \(f(x^k) \to \nu(\sBar) \) (the function values converge to the value of the closest feasible problem),
    \item \( C(x^k) - y^k \to \sBar  \) (the constraint transgression is asymptotically minimized),
    \item \( \left\Vert  y^k  - \Proj_\KK ( C(x^k) ) \right\Vert \to 0 \) (the slack variable \(y^k\) asymptotically coincides with the projection of \(C(x^k)\) onto \(\KK\)). 
\end{itemize}
We say that IALM \textbf{semi-quantitatively converges} if it simply converges and furthermore 
\begin{itemize}
    \item \(\left\Vert C(x^k) - y^k -\sBar \right\Vert = O\left(\tfrac{1}{\sqrt{\sum_{i=\changed{0}}^{\changed{k-1}} \gamma_i}} \right) \),
    \item \( \left\Vert  y^k  - \Proj_\KK ( C(x^k) ) \right\Vert = O\left(\tfrac{1}{\sqrt{\sum_{i=\changed{0}}^{\changed{k-1}} \gamma_i}} \right)\).
\end{itemize}
We say that IALM \textbf{quantitatively converges} if it semi-quantitatively converges and furthermore 
\begin{itemize}
    \item \( \left\vert f(x^k) - \nu(\sBar) \right\vert = O\left(\tfrac{1}{\sqrt{\sum_{i=\changed{0}}^{\changed{k-1}} \gamma_i}} \right) \).
\end{itemize}
\end{definition}}

\changed{Before tackling the convergence of IALM, we first clarify the meaning of lower semicontinuity and subdifferentiability of the value function \(\nu\) of the shifted problem,
 as these notions are used in \Cref{thm: nu lsc implies convergence} and \Cref{thm: nu subdifferentiable implies convergence}. The following two statements, due to \cite[Theorem 15]{rockafellar1974conjugate}, hold:
\begin{itemize}
    \item \(\nu\) is lower-semicontinuous at a given point \(\tilde s\) if and only if there is zero duality gap for the \(\tilde s\)-shifted Problem \eqref{eq:convex_min_problem_shifted}.
    \item \(\nu\) is subdifferentiable at a given point \(\tilde s\) if and only if there exists a Kuhn-Tucker vector associated with the \(\tilde s\)-shifted Problem \eqref{eq:convex_min_problem_shifted}, or equivalently, if the dual of this shifted problem has a minimizer.
\end{itemize}
These classical results are recalled solely for interpretative purposes.}

\deleted{However, classical theorems such as those depending on constraint qualification typically do not help to establish the strong duality of the \(\sBar\)-shifted problem when the original Problem \eqref{eq:convex_min_problem} is infeasible. 
Indeed the shift \(\sBar\) renders the shifted problem only marginally feasible, which typically implies that the \(\sBar\)-shifted constraint set has empty interior.}

\subsection{Convergence of IALM to the closest feasible problem}

The following theorem states that IALM always converges in value and constraint transgression. 

\begin{theorem} \label{thm:contraint transgression and value convergence}
    Let \( (x^k, y^k, \lambda^k, s^k )_{k \in \N} \) be a sequence generated by an IALM (\Cref{algo:IALM}) associated with problem~\eqref{eq:convex_min_problem} 
    \changed{with errors \((\varepsilon_k)_{k\in\N^*}\) and penalty parameters \((\gamma_k)_{k\in\N}\) satisfying \Cref{assumption:errors_and_step_sizes}}.
     it holds that:
\begin{itemize}
    \item[(a)] \(  C(x^k) - y^k \to \sBar = \argmin_{s \in \cl(\Scal)} \left\Vert s \right\Vert \),
    \item[(b)] \(\left\Vert  y^k  - \Proj_\KK ( C(x^k) ) \right\Vert \to 0 \), 
    \item[(c)] \(f(x^k) \to h^*(\sBar)\).
\end{itemize}
Furthermore if \(h^*(\sBar) < \infty\) (which is true in particular if the closest feasible problem is defined since \( h^*\) is the closure of \(\nu\)) the following holds:
\begin{itemize}
    \item[(d)] \(\left\Vert C(x^k) - y^k -\sBar \right\Vert = O\left(\tfrac{1}{\sqrt{\sum_{i=\changed{0}}^{\changed{k-1}} \gamma_i}} \right) \),
    \item[(e)] \(\left\Vert  y^k  - \Proj_\KK ( C(x^k) ) \right\Vert = O\left(\tfrac{1}{\sqrt{\sum_{i=\changed{0}}^{\changed{k-1}} \gamma_i} }\right)\).
\end{itemize}
\end{theorem}

\begin{proof}

\Cref{prop:approx aug Lagrangian is approx prox point} implies that the results on IPPA applied to the dual function are applicable to IALM.
\Cref{prop:sk_convergence} yields the convergence of \(s^k\), where \Cref{lem:set_inclusions} is used to ensure that the definition of \(\sBar = \argmin_{s\in \cl(\dom(\partial h))} \Vert s \Vert \) corresponds to the definition \( \sBar = \argmin_{s\in \cl(\Scal)} \Vert s \Vert\).
 \Cref{thm:convergence_h_s_k_to_h_s_bar} and \Cref{prop:h*(sk) close to f(xk)} allow us to conclude that \(f(x^k) \to h^*(\sBar)\), with the desired convergence rates when \(h^*(\sBar) < \infty\).

Let us now show that \( \left\Vert  y^k  - \Proj_\KK ( C(x^k) ) \right\Vert \to 0 \).
We have
\begin{align}
    \left\Vert C(x^k) - \Proj_\KK(C(x^k)) - \sBar \right\Vert^2 &= \begin{multlined}[t] \left\Vert C(x^k) - \Proj_\KK(C(x^k)) \right\Vert^2 \\ - 2 \inner{\sBar}{C(x^k) - \Proj_\KK(C(x^k))} + \left\Vert \sBar  \right\Vert^2 \end{multlined} \nonumber \\
    & \leq \left\Vert C(x^k) - \Proj_\KK(C(x^k)) \right\Vert^2  - \left\Vert \sBar  \right\Vert^2 \nonumber \\
    & \added{ (\text{Using } C(x^k) - \Proj_\KK(C(x^k)) \in \Scal } \nonumber \\
    & \added{ \text{ so } \inner{ \sBar - 0}{C(x^k) - \Proj_\KK(C(x^k)) - \sBar} \geq 0) } \nonumber \\
    & \leq \left\Vert C(x^k) - y^k \right\Vert^2  - \left\Vert \sBar  \right\Vert^2 \nonumber \\
    & \added{ (\text{because } y^k \in \KK ) } \nonumber \\
    & \leq \left\Vert s^k \right\Vert^2  - \left\Vert \sBar  \right\Vert^2 . \label{eq:constraint transgression convergence}
\end{align}
And also
\begin{align}
    \left\Vert y^k - \Proj_\KK(C(x^k)) \right\Vert^2 &= \left\Vert y^k - C(x^k) - \sBar \right\Vert^2 + \left\Vert C(x^k) - \Proj_\KK(C(x^k)) -\sBar \right\Vert^2 \nonumber \\
    &  \quad \quad - 2 \inner{y^k - C(x^k) - \sBar}{C(x^k) - \Proj_\KK(C(x^k)) -\sBar} \nonumber \\
    & \leq \left( \, \left\Vert y^k - C(x^k) - \sBar \right\Vert +  \left\Vert C(x^k) - \Proj_\KK(C(x^k)) -\sBar \right\Vert \, \right)^2 \nonumber \\
    & \leq \left( \, \left\Vert s^k - \sBar \right\Vert +  \left\Vert C(x^k) - \Proj_\KK(C(x^k)) -\sBar \right\Vert \, \right)^2 \nonumber \\
    & \leq \left( \, \left\Vert s^k - \sBar \right\Vert +  \sqrt{\left\Vert s^k \right\Vert^2  - \left\Vert \sBar  \right\Vert^2} \, \right)^2, \label{eq:constraint transgression convergence 2} 
\end{align}
where the last inequality follows from \eqref{eq:constraint transgression convergence}.
Using \(  \left\Vert s^k - \sBar \right\Vert \to 0\) and \( \left\Vert s^k \right\Vert^2  - \left\Vert \sBar  \right\Vert^2 \to 0\) (see \Cref{prop:sk_convergence}), it follows that \(\left\Vert y^k - \Proj_\KK(C(x^k)) \right\Vert \to 0\).

Furthermore, if \(h^*(\sBar) < \infty \), \Cref{prop:sk_convergence} guarantees that 
\[ \left\Vert s^k - \sBar \right\Vert = O\left(\tfrac{1}{\sqrt{\sum_{i=\changed{0}}^{\changed{k-1}} \gamma_i}} \right)
\text{ and } \left\Vert s^k \right\Vert^2  - \left\Vert \sBar  \right\Vert^2 = O\left(\tfrac{1}{\sum_{i=\changed{0}}^{\changed{k-1}} \gamma_i} \right) \]
 which with \eqref{eq:constraint transgression convergence 2} implies that
\(\left\Vert y^k - \Proj_\KK(C(x^k)) \right\Vert^2 = O\left(\tfrac{1}{\sum_{i=\changed{0}}^{\changed{k-1}} \gamma_i} \right) \) thereby reaching the desired conclusion. 
\end{proof}

The limit of the iterates of IALM, however, can correspond to the closest feasible problem as the following corollary highlights. 

\begin{corollary} \label{thm: nu lsc implies convergence}
    Let \( (x^k, y^k, \lambda^k, s^k)_{k \in \N} \) be a sequence generated by an IALM (\Cref{algo:IALM}) associated with Problem \eqref{eq:convex_min_problem} 
    \changed{with errors \((\varepsilon_k)_{k\in\N^*}\) and penalty parameters \((\gamma_k)_{k\in\N}\) satisfying \Cref{assumption:errors_and_step_sizes}}. 
    The algorithm simply converges to the closest feasible problem in the sense of \Cref{def:convergence_closest_feasible} if and only if the value function \(\nu\) is lower-semicontinuous and finite at \(\sBar\).
    If furthermore the closest feasible problem has finite value (i.e., \(\nu(\sBar) < \infty\)), then the algorithm semi-quantitatively converges to the closest feasible problem in the sense of \Cref{def:convergence_closest_feasible}. 
\end{corollary}

\begin{proof}
\(\nu\) is lower-semicontinuous at \(\sBar\) if and only if \(h^*(\sBar)\) = \(\nu(\sBar)\) and \Cref{thm:contraint transgression and value convergence} provides the desired result.
\end{proof}

\changed{The convergence to the closest feasible problem defined in \Cref{def:convergence_closest_feasible} says nothing about the behavior of the sequence \( (x_k)_{k\in \N}\), 
the following \Cref{prop:weak_accumulation_points_in_Sol_sBar} and \Cref{thm:boundedness of the xk implies convergence} explain what can be said about it. }

\changed{
\begin{proposition}  \label{prop:weak_accumulation_points_in_Sol_sBar}
    Let \( ((x^k, y^k, \lambda^k, s^k ))_{k \in \N} \) be a sequence generated by an IALM (\Cref{algo:IALM}) associated with Problem \eqref{eq:convex_min_problem}. 
    Any weak accumulation point of \((x^k)_{k\in\N}\) is in the solution set \(\mathrm{Sol}_{\sBar}\) of the closest feasible Problem \eqref{eq:bilevel optimization problem}.
\end{proposition}
\begin{proof}
    \Cref{thm:contraint transgression and value convergence} states that \(f(x^k)\) tends to a specific limit \(h^*(\sBar)\) in \( \R \cup \{+\infty\} \).
    Suppose that there exists a weak accumulation point \(\tilde x\) of \( (x^k)_{k\in\N} \). \(\X\) is closed and convex, so it is weakly closed and hence \(\tilde x\) must be in \(\X\).
    Let us define \( (x^{\phi(k)})_{k\in \N}\) a subsequence that converges weakly to \(\tilde x \in \X\).
    By \Cref{prop:sk_convergence}, we also have \( s^{\phi(k)} \to \sBar \), therefore the sequence \( (x^{\phi(k)}, s^{\phi(k)})_{k\in \N} \) weakly converges to \( (\tilde x, \sBar) \in \X \times \Y \) .
    \(\CC\) is closed and convex, so it is weakly closed, and since for all \( k \in \N \), \( (x^{\phi(k)}, s^{\phi(k)}) \in \CC \), we have \( (\tilde x, \sBar) \in \CC\).
    This in turn implies that \( f(\tilde x) \geq \nu(\sBar)\) by definition of \(\nu(\sBar)\). 
    But \(f\) is lower-semicontinuous so we also have the inequality at the limit \(f(\tilde x ) \leq \lim f(x^{\phi(k)}) = h^*(\sBar) \leq \nu(\sBar)\). 
    Those two inequalities imply the equality \( \nu(\sBar) = f(\tilde x) =  h^*(\sBar) \) which shows that, if there exists at least one accumulation point, \(\nu\) is lower-semicontinuous at \(\sBar\) 
    (this will be used in \Cref{thm:boundedness of the xk implies convergence}).
    \( \nu(\sBar) = f(\tilde x) \) also means that \( \tilde x \in \mathrm{Sol}_{\sBar}\). 
    Since \( \tilde x\) was chosen arbitrarily in the set of weak accumulation point of \( (x^k)_{k\in\N} \), we obtain that all of its weak accumulation points are in \( \mathrm{Sol}_{\sBar} \).
\end{proof}
} 

The following \Cref{thm:boundedness of the xk implies convergence} shows that the convergence to the closest feasible problem can only fail when the iterates \((x^k)_{k\in\N}\) diverge.
It can also be used to justify a heuristic, namely if the iterates \((x^k)_{k\in\N}\) are deemed to have converged, which indicates convergence to the closest feasible problem.

\changed{
\begin{proposition} \label{thm:boundedness of the xk implies convergence}
    Let \( (x^k, y^k, \lambda^k, s^k)_{k \in \N} \) be a sequence generated by an IALM (\Cref{algo:IALM}) associated with Problem \eqref{eq:convex_min_problem}
    with errors \((\varepsilon_k)_{k\in\N^*}\) and penalty parameters \((\gamma_k)_{k\in\N}\) satisfying \Cref{assumption:errors_and_step_sizes}.
     If the sequence \((x^k)_{k\in\N}\) is bounded:
    \begin{itemize}
        \item   \(\mathrm{Sol}_{\sBar}\) is nonempty and we have \( \mathrm{dist}(x^k, \mathrm{Sol}_{\sBar}) \to 0 \),
        \item \(\nu\) is lower-semicontinuous at \(\sBar\) and we have simple or semi-quantitative convergence to the closest feasible problem by applying \Cref{thm: nu lsc implies convergence}.
    \end{itemize}
\end{proposition}
\begin{proof}
    By hypothesis \((x^k)_{k\in\N}\) is bounded so it has at least one weak accumulation point.
    Therefore, according to \Cref{prop:weak_accumulation_points_in_Sol_sBar}, \(\mathrm{Sol}_{\sBar}\) is nonempty and contains all the weak accumulation points of \( (x^k)_{k\in\N} \).
   Finally, the boundedness of \( (x^k)_{k\in\N} \) implies that its distance to the set of its accumulation points tends to \(0\), in particular we have \( \mathrm{dist}(x^k, \mathrm{Sol}_{\sBar}) \to 0 \).
    As noted in the proof of \Cref{prop:weak_accumulation_points_in_Sol_sBar}, if there is at least one accumulation point of \( (x^k)_{k\in\N} \), then \(\nu\) is lower-semicontinuous.
\end{proof}
}

\deleted{We illustrate how this theorem is applicable for semidefinite programming in \Cref{example:SDP}.}
\changed{Verifying the hypothesis in \Cref{thm:boundedness of the xk implies convergence} or \Cref{thm: nu lsc implies convergence} on a given problem can be difficult in practice. 
The following \Cref{thm:level boundedness locally uniformly implies lsc} provides a more convenient set of sufficient assumptions to obtain convergence to the closest feasible problem. 
It depends on two notions which we define before stating \Cref{thm:level boundedness locally uniformly implies lsc}.}

\changed{
\begin{definition} \label{def:level_boundedness_locally_uniformly}
    The function \(f\) is said to be level bounded with respect to constraint perturbations if for every \(s \in \Y\), there exists a neighborhood \(V_s\) of \(s\) such that \(f\)
     is level bounded when restricted to the set \( \{ x \in \X \mid C(x) \in \KK + V_s \} \).
\end{definition}
} 

\changed{
    In practice, level boundedness with respect to constraint perturbations can be obtained from sufficient conditions such as having \(f\) level bounded
     or having the set \( \{ x \in \X \mid \mathrm{dist}(C(x), \KK) \leq R \} \) itself bounded for all \(R \geq 0\).
      In the finite-dimensional setting, this condition can be equivalently stated in terms of recession directions.
      Since the constraint set might be empty, we formalize a (natural) generalization of the notion of recession direction for the constraint.
}

\begin{definition} \label{def:def_recession_dir_constraints}
    Suppose that \(\X\) and \(\Y\) are finite-dimensional. Consider the constraints of Problem \eqref{eq:convex_min_problem} written as \(C(x) \in \KK\).
     The recession cone of the constraints (or recession directions of the constraints) \(\D\) 
     is defined as 
\begin{equation*}
    \D = \{ d \; | \; \exists s \in \Y, \exists x \in \X , \forall t \leq 0 , C(x + td)  \in \KK + s\}
\end{equation*}
\end{definition}

\changed{
In \Cref{example:ALM_for_inequality_constrained_convex_optimization} for instance,
the recession directions of the constraints are simply the recession directions in common to all functions \( c_1 , \dots, c_m\).}
    
\begin{proposition} \label{prop:rec_dir_imply_level_bounded}
    \changed{
    If \(\X\) and \(\Y\) are finite-dimensional, then the objective function \(f\) has no recession direction in common with the constraints (as defined in \Cref{def:def_recession_dir_constraints}) if
    and only if \(f\) is level bounded with respect to constraint perturbations (as defined in \Cref{def:level_boundedness_locally_uniformly}). }
\end{proposition}

\begin{proof}
In \Cref{def:level_boundedness_locally_uniformly}, \(V_s\) can trivially be taken bounded with
no loss of generality. 
Consider the set \( \{ x \in \X \mid C(x) \in \KK + V_s \} \) for some bounded neighborhood \(V_s\) of some point \(s\in \Scal\).
The recession directions of that set are exactly the recession directions of the set 
\(\CC + (0, V_s)\) of the form \((d,0)\) (this follows directly from the definition of \(\CC + (0, V_s)\)). 
But since \(V_s\) is bounded those correspond to the recession directions of \(\CC \) of the form \((d,0)\).
Likewise, for any given \(\tilde s \in \Scal\), the recession directions of \(\{ d \; | \; \exists s \in \Y, \exists x \in \X , \forall t \leq 0 , C(x + td)  \in \KK + \tilde s\}\)
are also exactly the recession directions of \(\CC\) of the form \((d,0)\) (by definition of \(\CC\)).
We have just shown that the recession directions of  \( \{ x \in \X \mid C(x) \in \KK + V_s \} \) are exactly the 
recession directions of the constraints.
Therefore, in finite dimension, \(f\) is level bounded on \( \{ x \in \X \mid C(x) \in \KK + V_s \} \) if and only
if it shares no recession direction in common with the constraints. 
\end{proof}

\begin{theorem} \label{thm:level boundedness locally uniformly implies lsc}
Let \( (x^k, y^k, \lambda^k, s^k)_{k \in \N} \) be a sequence generated by an IALM (\Cref{algo:IALM}) associated with Problem \eqref{eq:convex_min_problem} 
\changed{with errors \((\varepsilon_k)_{k\in\N^*}\) and penalty parameters \((\gamma_k)_{k\in\N}\) satisfying \Cref{assumption:errors_and_step_sizes}}.
If the closest feasible problem has finite value (i.e., \(\nu(\sBar) < \infty\)), and either 
\begin{itemize}
    \item \(f\) is level bounded with respect to constraint perturbations (see \Cref{def:level_boundedness_locally_uniformly})
    \item or \(\X\) and \(\Y\) are finite-dimensional and \(f\) has no recession direction in common with the constraints,
\end{itemize}
then :
\begin{itemize}
    \item the algorithm semi-quantitatively converges to the closest feasible problem in the sense of \Cref{def:convergence_closest_feasible},
    \item  \added{\(\mathrm{Sol}_{\sBar} \), the solution set of closest feasible Problem \eqref{eq:bilevel optimization problem}, is nonempty and bounded and we have \( \mathrm{dist}(x^k, \mathrm{Sol}_{\sBar}) \to 0 \)}
\end{itemize}

\end{theorem}

\begin{proof}
As noted in \Cref{prop:rec_dir_imply_level_bounded}, the condition on recession directions implies that \(f\) is level bounded with respect to constraint perturbations. 

Let \(V_{\sBar}\) be a neighborhood of \(\sBar\) such that \(f\) is level bounded on \( \{ x \in \X \mid C(x) \in \KK + V_{\sBar} \} \). Since \(s^k \to \sBar\), for \(k\) sufficiently large we have \( s^k \in V_{\sBar}\).
Since \(f(x^k) \to h^*(\sBar) \leq \nu(\sBar) < \infty\), the sequence \((f(x^k))_{k\in\N}\) is upper-bounded. Therefore, for \(k\) sufficiently large, the sequence \((x^k)_{k \in \N}\) stays in 
a sublevel set of \(f\) on the set \( \{ x \in \X \mid C(x) \in \KK + V_{\sBar} \} \) which implies that \((x^k)_{k \in \N}\) is bounded.
We can apply \Cref{thm:boundedness of the xk implies convergence},
along with the fact that we assumed that the closest feasible problem has finite value
 to obtain the semi-quantitative convergence and \( \mathrm{dist}(x^k, \mathrm{Sol}_{\sBar}) \to 0 \).
\end{proof}

\deleted{
The condition of level boundedness locally uniformly can seem difficult to verify, it can however simply be obtained from sufficient conditions such as having \(f\) level bounded or having the shifted constraints always bounded by a term that depends continuously on the norm of the shift. These two cases are probably the most suitable for applying \Cref{thm:level boundedness locally uniformly implies lsc} in practice. 
}


\subsection{Quantitative convergence to the closest feasible problem} \label{sec:quantitative convergence to the closest feasible problem}

When \(\nu\) is not only lower-semicontinuous but also subdifferentiable at \(\sBar\), then quantitative convergence to the closest feasible problem can be ensured.

\begin{theorem} \label{thm: nu subdifferentiable implies convergence}
    Let \( (x^k, y^k, \lambda^k, s^k)_{k \in \N} \) be a sequence generated by an IALM (\Cref{algo:IALM}) associated with Problem \eqref{eq:convex_min_problem} 
    \changed{with errors \((\varepsilon_k)_{k\in\N^*}\) and penalty parameters \((\gamma_k)_{k\in\N}\) satisfying \Cref{assumption:errors_and_step_sizes}}.
     If the value function \(\nu\) is subdifferentiable at \(\sBar\) then the algorithm converges quantitatively to the closest feasible problem in the sense of \Cref{def:convergence_closest_feasible}. 
\end{theorem}

\begin{proof}
Using \Cref{prop:approx aug Lagrangian is approx prox point} we can apply results of the IPPA on the dual function \(h\) to our IALM iterates.
Semi-quantitative convergence follows from \Cref{thm: nu lsc implies convergence} as subdifferentiability implies \(\nu(\sBar)\) is finite. 
\Cref{thm:convergence rate h^*(s^k) - h^*(sBAR)} states that \( | h^*(s_\star^k) - h^*(\sBar) | = O\left(\tfrac{1}{\sqrt{\sum_{i=\changed{0}}^{\changed{k-1}} \gamma_i}} \right)\) 
and \Cref{prop:h*(sk) close to f(xk)} with \ref{(A2)} states that \( | h^*(s_\star^k) - f(x^k) | = O\left(\tfrac{1}{\sqrt{\sum_{i=\changed{0}}^{\changed{k-1}} \gamma_i}} \right)\). 
Since subdifferentiability of \(\nu\) at \(\sBar\) implies it is lsc and finite there, \(h^*(\sBar) = \nu(\sBar)\).
 Thus, by triangle inequality, we get the result \( \left\vert f(x^k) - \nu(\sBar) \right\vert = O\left(\tfrac{1}{\sqrt{\sum_{i=\changed{0}}^{\changed{k-1}} \gamma_i}} \right) \) which proves quantitative convergence to the closest feasible problem.
\end{proof}

Showing the subdifferentiability of the value function \(\nu\) can itself be a difficult task, see \Cref{sec:related_works} for pointers to works on the analytic study of the value function. The following corollary provides a sufficient condition for the subdifferentiability of \(\nu\) in the case of polyhedral constraints.

\begin{corollary} \label{cor:polhedral_constraints_Lipschitz_objective_imply_subdifferentiability_nu}
Let \( (x^k, y^k, \lambda^k, s^k)_{k \in \N} \) be a sequence generated by an IALM (\Cref{algo:IALM}) associated with Problem \eqref{eq:convex_min_problem} 
\changed{with errors \((\varepsilon_k)_{k\in\N^*}\) and penalty parameters \((\gamma_k)_{k\in\N}\) satisfying \Cref{assumption:errors_and_step_sizes}}.
If the mapping \(C\) and the sets \(\KK\) and \(\X\) are polyhedral and the objective function \(f\) is Lipschitz continuous, then the algorithm converges quantitatively to the closest feasible problem in the sense of \Cref{def:convergence_closest_feasible}.
\end{corollary}

\begin{proof}
This proof consists of demonstrating that the value function \(\nu\) is Lipschitz continuous, and therefore subdifferentiable, under the given conditions.

The Hausdorff distance \(\Delta (A,B)\) between two sets \(A\) and \(B\) is defined as:
\[
    \Delta (A,B) = \max \left\{ \sup_{a\in A} \inf_{b\in B} \left\Vert a - b \right\Vert , \sup_{b\in B} \inf_{a\in A} \left\Vert a - b \right\Vert \right\}.
\]\deleted{
Our first step is to show that the set-valued mapping \( s \mapsto \CC_{s} = \{ x \in \X \mid (x,s)\in \CC \}\) is Lipschitz continuous with respect to the Hausdorff distance. This means there exists a constant \(\kappa > 0\) such that for any shifts \(s, r \in \Y\):
\begin{equation*}
    \Delta ( \CC_{s}, \CC_{r} ) \leq \kappa \left\Vert s - r \right\Vert
\end{equation*}
Given that the mapping \(C\) and the sets \(\X\) and \(\KK\) are polyhedral, the set \(\CC = \{ (x,s) \in \X \times \Y \mid s \in C(x) - \KK \}\) is also polyhedral. This implies that \(\CC_s\) can be described as the intersection of a finite number of half-spaces as  \(\CC_{s} = \{ x \in \X \mid \langle a_i, x \rangle + \langle b_i , s \rangle \leq c_i , \text{ for } i = 1, \ldots, n \} \) for some \( (a_1 , \dots , a_n) \) in the ambient space of \(\X\) and \( (b_1 , \dots, b_n) \in \Y^n\) and \( (c_1,  \dots, c_n) \in \R^n\). 
We now explain why it suffices to consider the case where \(\X\) and \(\Y\) are finite-dimensional to show the Lipschitz continuity of \( s \mapsto \CC_{s}\).
In the definition \(\CC_{s} = \{ x \in \X \mid \langle a_i, x \rangle + \langle b_i , s \rangle \leq c_i , \text{ for } i = 1, \ldots, n \} \),
any vector \(\tilde x \) orthogonal to the vectors \( a_1, a_2, \ldots, a_n\) can be added to or subtracted from \(x\) without playing any role at all,
therefore we can safely ignore these components and consider the subspace of \(\X\) generated by \( a_1, a_2, \ldots, a_n\) instead of the entire space \(\X\).
Similarly for \(s\), in the definition of \(\CC_{s} = \{ x \in \X \mid \langle a_i, x \rangle + \langle b_i , s \rangle \leq c_i , \text{ for } i = 1, \ldots, n \} \) any vector \(\tilde s \) orthogonal to the vectors \( b_1, b_2, \ldots, b_n\) can be added to or subtracted from \(s\) without playing any role at all. 
If the inequality \(\Delta ( \CC_{s}, \CC_{r} ) \leq \kappa \left\Vert s - r \right\Vert\) holds for any \(s, r\) in the vector space generated by \( b_1, b_2, \ldots, b_n\), then it holds for any \(s, r\) in \(\Y\) because \(\CC_{s}\) and \(\CC_{r}\) are unchanged by components of \(s, r\) that would be orthogonal to the vector space generated by \( b_1, b_2, \ldots, b_n\) and \(\left\Vert s - r \right\Vert\) can only get larger by adding such components. Therefore, considering the finite-dimensional case is sufficient to show the Lipschitz continuity of \( s \mapsto \CC_{s}\).
The finite-dimensional case is treated in \cite[Example 9.35]{rockafellar2009variational}. }
\cite[Theorem 2.207]{bonnans2000perturbation} states that the convex polyhedral multifunction \( s \mapsto \CC_{s} = \{ x \in \X \mid (x,s)\in \CC \}\) is Lipschitz continuous with respect to the Hausdorff distance.
Let \(\kappa > 0\) be a Lipschitz constant that verifies
\begin{equation} \label{eq:Hausforff_dist_is_lipschitz}
    \forall s, r \in \Scal, \; \Delta ( \CC_{s}, \CC_{r} ) \leq \kappa \left\Vert s - r \right\Vert.
\end{equation}
Let \(L_f\) be the Lipschitz constant of \(f\). For any \(x_s \in \CC_s\) and \(x_{r} \in \CC_r\), the Lipschitz continuity of \(f\) implies \( f(x_s) - f(x_r) \leq L_f \Vert x_s - x_r \Vert \).
By the definition of \(\nu(s) = \inf_{x \in \CC_s} f(x)\), we have \(\nu(s) \leq f(x_s)\). Thus,
\[
    \nu(s) - f(x_r) \leq L_f \Vert x_s - x_r \Vert .
\]
We can choose \(x_s \in \CC_s\) such that \(\Vert x_s - x_r \Vert = \inf_{x \in \CC_s} \Vert x - x_r \Vert\). By the definition of the Hausdorff distance, this infimum is less than or equal to \(\Delta(\CC_s, \CC_r)\). Therefore,
\[
    \nu(s) - f(x_r) \leq L_f \Delta ( \CC_{s}, \CC_{r} ).
\]
Since this holds for any \(x_r \in \CC_r\), we can take the infimum over \(x_r \in \CC_r\) on the left side:
\[
    \nu(s) - \inf_{x_r \in \CC_r} f(x_r) \leq L_f \Delta ( \CC_{s}, \CC_{r} ),
\]
which means
\[
    \nu(s) - \nu(r) \leq L_f \Delta ( \CC_{s}, \CC_{r} ).
\]
Combining this with the Lipschitz continuity of \(\CC_s\) (Equation \eqref{eq:Hausforff_dist_is_lipschitz}), we get:
\[ \nu(s) - \nu(r) \leq L_f \kappa \left\Vert s - r \right\Vert. \]
By symmetry, we can swap \(s\) and \(r\) to obtain \(\nu(r) - \nu(s) \leq L_f \kappa \left\Vert r - s \right\Vert\). Together, these imply:
\[ | \nu(s) - \nu(r) | \leq L_f \kappa \left\Vert s - r \right\Vert. \]
Thus, \(\nu\) is Lipschitz continuous on its domain, which implies that \(\nu\) is subdifferentiable on its domain. 
\(\sBar\) is clearly in the domain of \(\nu\) since the polyhedral constraints imply that there exists and \(\overline x \) such that \((\overline x , \sBar) \in \CC\) and \(f\) is finite on \(\X \) (since it is Lipschitz and proper) so \(\nu(\sBar) \leq f(\overline x) < \infty \), so \(\nu\) is subdifferentiable at \(\sBar\).  
\Cref{thm: nu subdifferentiable implies convergence} then provides the result.
\end{proof}

The Lipschitz continuity hypothesis for \(f\) is quite restrictive; the following corollary requires Lipschitzness only on bounded sets instead, to allow a larger family of objective functions (such as, for instance, some quadratic objective functions). 

\begin{corollary} \label{thm:polyhedral_constraints_levelbounded_imply_nu_subdiff} 
    Let \( (x^k, y^k, \lambda^k, s^k )_{k \in \N} \) be a sequence generated by an IALM (\Cref{algo:IALM}) associated with Problem \eqref{eq:convex_min_problem} 
    \changed{with errors \((\varepsilon_k)_{k\in\N^*}\) and penalty parameters \((\gamma_k)_{k\in\N}\) satisfying \Cref{assumption:errors_and_step_sizes}}. 
    If the mapping \(C\) and the sets \(\KK\) and \(\X\) are polyhedral, the objective function \(f\) is Lipschitz on any bounded set in \(\X\) and either 
    \begin{itemize}
        \item \(f\) is level bounded with respect to constraint perturbations (as defined in \Cref{def:level_boundedness_locally_uniformly}),
        \item or \(\X\) and \(\Y\) are finite-dimensional, and \(f\) has no recession direction in common with the constraints (as defined in \Cref{def:def_recession_dir_constraints}),
    \end{itemize}
    then:
    \begin{itemize}
        \item the algorithm converges quantitatively to the closest feasible problem in the sense of \Cref{def:convergence_closest_feasible}
        \item \(\mathrm{Sol}_{\sBar} \), the solution set of the bilevel optimization Problem \eqref{eq:bilevel optimization problem}, is bounded and \changed{\( \mathrm{dist}(x^k, \mathrm{Sol}_{\sBar}) \to 0 \).} 
    \end{itemize}
\end{corollary}

\begin{proof}
The proof is very similar to the proof of \Cref{cor:polhedral_constraints_Lipschitz_objective_imply_subdifferentiability_nu}.
 The only difference is that \(f\) is not assumed to be Lipschitz on the entire space but only on bounded sets.
 \((x^k)_{k\in\N}\) is bounded, which makes the Lipschitzness of \(f\) on bounded sets sufficient. 

\Cref{thm:level boundedness locally uniformly implies lsc} establishes the second claim, in particular the iterates \((x^k)_{k\in\N}\) are bounded.
 Let us call \( \mathcal{B}\) a bounded convex set included in \(\X\) which contains the entire sequence \((x^k)_{k\in\N}\). As described in \cite[Theorem 1]{cobzas1978norm}, we can construct a convex closed function \(\tilde f\) equal to \(f\) on \( \mathcal{B}\) and globally Lipschitz as :
\begin{equation*}
    \tilde f (x) = \inf_{ \tilde x \in \mathcal{B}} \; f(\tilde x) + L_{f,\mathcal{B}} \Vert \tilde x - x \Vert 
\end{equation*}
where \(L_{f,\mathcal{B}} \) is a Lipschitz constant of \(f\) on \(\mathcal{B}\). 
The iterates of IALM on \eqref{eq:convex_min_problem} are also iterates of IALM on the same problem where \(f\) has been replaced by \(\tilde f\), therefore \Cref{cor:polhedral_constraints_Lipschitz_objective_imply_subdifferentiability_nu} applies and provides the result.
\end{proof}  

\deleted{
\begin{corollary} \label{thm:recession_dir_polyhedral_constraints_imply_nu_subdiff}
    Let \( ((x^k, y^k, \lambda^k, s^k ))_{k \in \N} \) be a sequence generated by an IALM (\Cref{algo:IALM}) associated with Problem \eqref{eq:convex_min_problem} 
    \changed{with errors \((\varepsilon_k)_{k\in\N^*}\) and penalty parameters \((\gamma_k)_{k\in\N}\) satisfying \Cref{assumption:errors_and_step_sizes}}. 
    Suppose that the dimension of \(\X\) is finite.
    If the mapping \(C\) and the sets \(\KK\) and \(\X\) are polyhedral and the objective function \(f\) is finite everywhere on \(\X\) and has no recession direction in common with the constraints,
    then:
\begin{itemize}
    \item the algorithm converges quantitatively to the closest feasible problem in the sense of \Cref{def:convergence_closest_feasible}
    \item the sequence \((x^k)_{k\in\N}\) converges to the solution set of the bilevel optimization Problem \eqref{eq:bilevel optimization problem}.
\end{itemize}
\end{corollary}
\begin{proof}
The conditions on recession directions means that the function \( \psi : (x,s) \mapsto f(x) + \delta_\CC( x, s) \) is level bounded in \(x\) locally uniformly in \(s\) according to \cite[Theorem 3.31]{rockafellar2009variational}. Furthermore, in finite dimensions, a convex function is locally Lipschitz on its domain, so \(f\) is Lipschitz on every bounded sets and \Cref{thm:polyhedral_constraints_levelbounded_imply_nu_subdiff} gives us the result. 
\end{proof}
}

%% file: diagram_part3.tex
\begin{figure}[t]
\begin{center}
\resizebox{0.9\linewidth}{!}{%
   \begin{tikzpicture}[
      box/.style = {
        draw,
        rectangle,
        minimum width=3.0cm,
        minimum height=1.2cm,
        align=center
      },
      dashedbox/.style = {
        draw,
        rectangle,
        minimum width=3.0cm,
        minimum height=1.2cm,
        align=center,
        dashed
      },
      font=\large
    ]
    \NoHyper

    \node[dashedbox, text width=5cm] (always holds) at (8, 4) {
      Always holds.
    };

    \node[box, text width=5cm] (finite dim rec dir) at (0, 4) {
      Finite dimensions. Constraints and \(f\) have no common recession directions.
    };

    \node[box, text width=5cm] (f level bounded) at (0, 0) {
      \(f\) is level bounded with respect to constraint perturbations.
    };
    \draw[->] (finite dim rec dir.south) -- node[sloped, above] {\Cref{prop:rec_dir_imply_level_bounded}}  (f level bounded) ;

    \node[box, text width=5cm] (nu sbar  finite) at (0, -4) {
      The closest feasible problem has finite value: \(\nu(\sBar) < \infty\).
    };

    \node[box, text width=5cm] (f lipschitz on bounded sets) at (0, -6) {
      \(f\) is Lipschitz on any bounded set.
    };
   
    \node[box, text width=5cm] (polyhedral constraints) at (0, -8) {
      The constraints are polyhedral.
    };

    \node[box, text width=5cm] (f lipschitz) at (0, -10) {
      \(f\) is Lipschitz.
    };

    \node[box, text width=5cm] (xk bounded) at (5.2, -2) {
      Sequence \((x^k)_{k\in\N}\) is bounded.
    };
    \node[draw, circle, minimum size=0.9cm] (plus 1) at (0,-2) {$+$};
    \draw[->] (f level bounded.south) + (-1,0) -- node[sloped, above] {}  (plus 1) ;
    \draw[->] (nu sbar  finite.north) + (-1,0) -- node[sloped, above] {}  (plus 1) ;
    \draw[->] (plus 1) -- node[sloped, above] {\Cref{thm:level boundedness locally uniformly implies lsc}} (xk bounded) ;

    \node[box, text width=5cm] (nu lsc) at (8, 2) {
      \(\nu\) is lower-semicontinuous at \(\sBar\).
    };
    \draw[->] (xk bounded) -- node[sloped, above] {\Cref{thm:boundedness of the xk implies convergence}}  (nu lsc) ;

    \node[box, text width=5cm] (nu subdifferentiable) at (8, -7.5) {
      \(\nu\) is subdifferentiable at \(\sBar\).
    };
    \node[draw, circle, minimum size=0.9cm] (plus 3) at (4,-6) {$+$};
    \draw[->] (f lipschitz on bounded sets)  -- node[sloped, above] {}  (plus 3) ;
    \draw[->] (xk bounded) -- node[sloped, above] {}  (plus 3) ;
    \draw[->] (polyhedral constraints.east) + (0,0.4) -- node[sloped, above] {}  (plus 3) ;
    \draw[->] (plus 3) -- node[sloped, above] {\Cref{thm:polyhedral_constraints_levelbounded_imply_nu_subdiff}} (nu subdifferentiable) ;
    \node[draw, circle, minimum size=0.9cm] (plus 4) at (4,-9) {$+$};
    \draw[->] (f lipschitz.east) -- node[sloped, above] {}  (plus 4) ;
    \draw[->] (polyhedral constraints.east) + (0,-0.4) -- node[sloped, above] {}  (plus 4) ;
    \draw[->] (plus 4) -- node[sloped, below] {\Cref{cor:polhedral_constraints_Lipschitz_objective_imply_subdifferentiability_nu}} (nu subdifferentiable) ;

    \node[box, text width=5cm] (always has limit) at (16, 4) {
    Function and constraint values, \(f(x^k)\) and \(C(x^k) - y^k\), have a limit.} ;
    \draw[->] (always holds) to node[sloped, above] {\Cref{thm:contraint transgression and value convergence}} (always has limit.west);

    \node[box, text width=5cm] (simple convergence) at (16, 1) { Simple convergence to the closest feasible problem.} ;
    \draw[->] (nu lsc) -- node[sloped, above] {\Cref{thm: nu lsc implies convergence}}  (simple convergence.west) ;

    \node[box, text width=5cm] (semi-quantitative convergence) at (16, -2) { Semi-quantitative convergence to the closest feasible problem.} ;
    \node[draw, circle, minimum size=0.9cm] (plus 2) at (10, 0) {$+$};
    \draw[->] (nu lsc) -- node[sloped, above] {}  (plus 2) ;
    %

    \draw[->, out=0, in=250] (nu sbar finite) to node[sloped, above] {} (plus 2);
    \draw[->] (plus 2) -- node[sloped, above] {\Cref{thm: nu lsc implies convergence}}  (semi-quantitative convergence.west) ;
    
    \node[box, text width=5cm] (xk to Sol sbar) at (16, -5) {
      Distance to solution set of closest feasible problem vanishes: \( \mathrm{dist}(x^k, \mathrm{Sol}_{\sBar}) \to 0 \).
    };
    \draw[->] (xk bounded.east) -- node[sloped, above] {\Cref{thm:boundedness of the xk implies convergence}}  (xk to Sol sbar.west) ;

    \node[box, text width=5cm] (quantitative convergence) at (16, -8) { Quantitative convergence to the closest feasible problem.} ;
    \draw[->] (nu subdifferentiable) -- node[sloped, above] {\Cref{thm: nu subdifferentiable implies convergence}}  (quantitative convergence.west) ;

    \endNoHyper

    \end{tikzpicture}
} 
\end{center}
\caption{Informal diagram summarizing the chain of implications between properties shown in \Cref{sec:convergence_ialm}. Arrows represent sufficient conditions. Plus signs represent the logical AND operator.}
\end{figure}

%% file: example_1.tex
\section{Examples} \label{sec:example_infeasible_problem}
This section provides examples that illustrate the applicability, implications, and meaning of the results and their hypotheses.  

\begin{example}[Different formulations yield different convergence properties].
In this example, we illustrate the case of linear constraints to highlight how different formulations of the same optimization problem yield different ALM and convergence properties. Let \(A,B\) be matrices, \(a,b\) be vectors of appropriate dimensions, and let \(f\) be a closed proper convex function that is finite everywhere.
     Assuming \(a\) is in the range of \(A\), consider the three formulations of the same optimization problem with linear equality and inequality constraints in finite dimensions:
\medskip

\begin{minipage}[t]{0.25\textwidth}
    \begin{itemize}
        \item[(a)] \( \begin{array}{rl}
            \displaystyle \min_{x \in \R^n} & f(x)         \\[1mm]
            \text{s.t.}& Ax = a , \\
            & Bx \leq b.
        \end{array} \)
    \end{itemize}
\end{minipage}
\hfill
\begin{minipage}[t]{0.30\textwidth}
    \begin{itemize}
        \item[(b)] \( \begin{array}{rl}
            \displaystyle \min_{x \in \R^n, Ax = a } & f(x)         \\[1mm]
            \text{s.t.} & Bx \leq b. \\
            & 
        \end{array} \)
    \end{itemize}
\end{minipage}
\hfill
\begin{minipage}[t]{0.34\textwidth}
    \begin{itemize}
        \item[(c)] \( \begin{array}{rl}
            \displaystyle \min_{x \in \R^n} & f(x)         \\[1mm]
            \text{s.t.} &  x \in \{ \tilde x | A\tilde x = a \} , \\
            & Bx \leq b .
        \end{array} \)
    \end{itemize}
\end{minipage}
\vspace{5pt}
\newline 
We assume that \(f\) has no recession direction in common with the constraints.
In all three cases (a), (b), and (c), \Cref{thm:polyhedral_constraints_levelbounded_imply_nu_subdiff} applies and provides quantitative convergence to the closest feasible problem. Meaning that the three algorithms converge respectively to their closest feasible problem and in all three cases the iterates \(x^k\) converge to the solution set of the closest feasible problem. The closest feasible problem is however defined differently in each cases as follows:

\begin{itemize}
    \item[(a)] \( \begin{array}{rl}
        \displaystyle \min_{x \in \R^n} & f(x)         \\[1mm]
        \text{s.t.} & x \in \argmin\limits_{\tilde x \in \R^n}  \left\| A\tilde x- a  \right\|^2 + \left\| \lfloor B\tilde x -b \rfloor_+ \right\|^2  \end{array} \)
\end{itemize}

\begin{itemize}
    \item[(b)] \( \begin{array}{rl}
        \displaystyle \min_{x \in \R^n} & f(x)         \\[1mm]
        \text{s.t.} & x \in \argmin\limits_{\tilde x \in \R^n, A\tilde x = a} \left\| \lfloor B\tilde x -b \rfloor_+ \right\|^2
    \end{array} \)
\end{itemize}

\begin{itemize}
    \item[(c)] \( \begin{array}{rl}
        \displaystyle \min_{x \in \R^n } & f(x)         \\[1mm]
        \text{s.t.} & x \in \argmin\limits_{\tilde x \in \R^n}  \Vert x - \Proj_{\{ \tilde x | A\tilde x = a \}} \left( x  \right) \Vert^2 + \left\| \lfloor B\tilde x -b \rfloor_+ \right\|^2 .
    \end{array} \)
\end{itemize}
\end{example}

\begin{example}[Infeasible QCQP, subdifferentiability of \(\nu\) is not guaranteed]. \label{example:QCQP_example}
Let us consider the following convex quadratically constrained quadratic programming (QCQP) where \(\alpha , \beta\) are fixed parameters:
\begin{equation} \label{eq:QCQP_example}
\begin{array}{rl}
\displaystyle \inf_{x\in \R} & -x\\[1mm]
\text{s.t.} & x^2 + \beta \leq 0\\[1mm]
& x + \alpha \leq 0 .
\end{array}
\end{equation}

Perhaps \(\alpha, \beta\) are parameters learned by an algorithm (as in machine learning) or are simply estimated by empirical methods.
 It is clear, in this example, that the constraints are infeasible if \(\alpha, \beta\) are not within a specific set and we may want to know how the algorithm behaves in this case.
  Assume that we perform the augmented Lagrangian method with constant step size on this problem, meaning that \(\forall k \in \N , \; \gamma_k = \gamma\).
The (exact) augmented Lagrangian method for this problem consists of the following iteration:
\begin{equation*} 
    \begin{array}{rl}
        \displaystyle x^{k+1} = & \begin{multlined}[t] \argmin_{x \in \R} \bigg( -x + \lambda^k (x^2 + \beta) + \mu^k (x + \alpha) + \tfrac{1}{2\gamma}\left( \lfloor x^2 + \beta \rfloor_+ \right)^2  \\ + \tfrac{1}{2\gamma} \left( \lfloor x + \alpha \rfloor_+ \right)^2 \bigg) \end{multlined} \\
        \lambda^{k+1} = & \lfloor \lambda^k + \gamma \left((x^{k+1})^2 - (x^k)^2\right) \rfloor_+ \\
        \mu^{k+1} = & \lfloor \mu^k + \gamma (x^{k+1} - x^k) \rfloor_+
    \end{array}
\end{equation*}

Previous works on the infeasible augmented Lagrangian method are not easily applicable: \cite{birgin2015optimality} is not applicable since the step sizes do not diverge, 
\cite{chiche2016augmented} is not applicable since this problem is not a QP and \cite{dai2023augmented} needs the subdifferentiability of the value function which is not guaranteed in this example, as we show. 
The value function \(\nu\) associated to \eqref{eq:QCQP_example} is: 
\[
 \begin{array}{rl}
\displaystyle \nu(s_1, s_2) \triangleq \inf_{x\in \R} & -x\\[1mm]
\text{s.t.} & x^2 + \beta + s_1\leq 0\\[1mm]
& x + \alpha + s_2 \leq 0.
\end{array}
\]

In this simple example, it is possible to assess subdifferentiability by hand. 
This exercise illustrates that assessing subdifferentiability of \(\nu\) is not straightforward even in this simple setting, and that nondifferentiability is not a rare phenomenon.

\begin{figure}[htbp]
    \centering
    \includegraphics[width=0.95\textwidth]{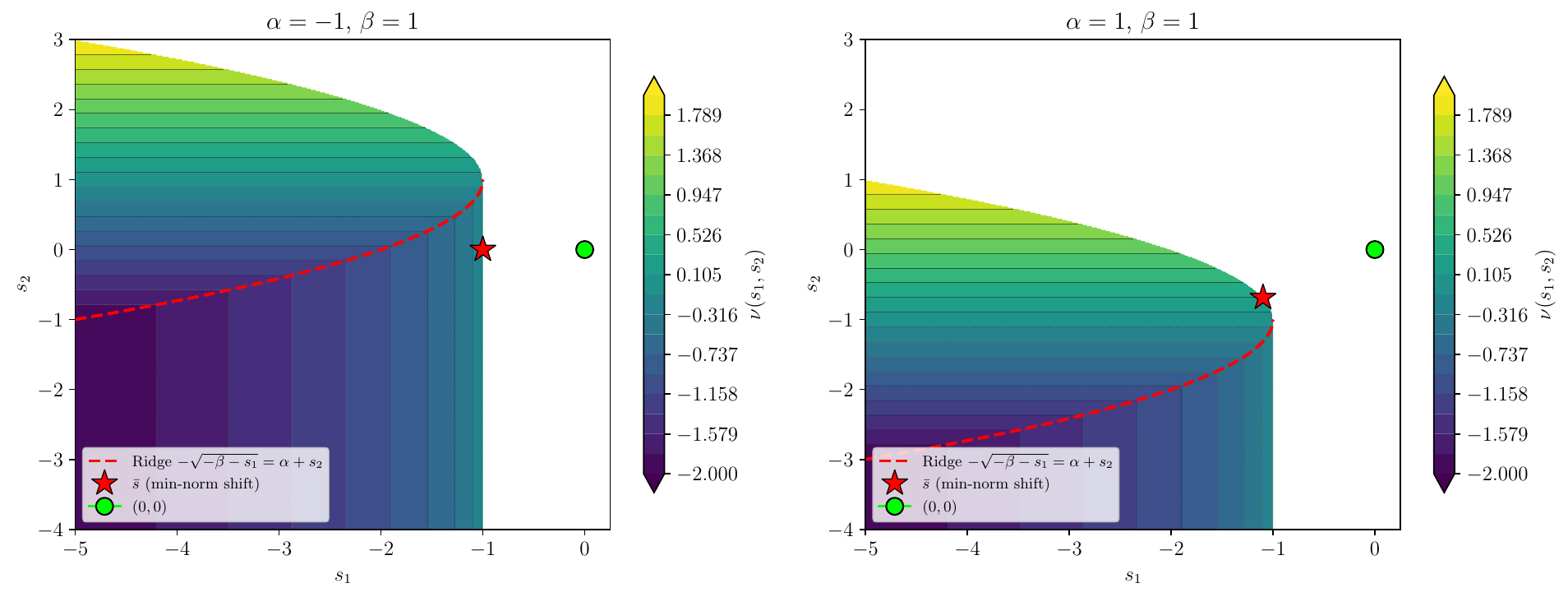}
    \caption{Contour plot of the value function $\nu(s_1, s_2)$ for the QCQP example. $\nu$ is nondifferentiable at $\sBar$ when 
    $(\alpha, \beta) = (- 1, 1)$ and subdifferentiable at $\sBar$ when $(\alpha, \beta) = (1, 1)$.}
    \label{fig:QCQP_nu_contour}
\end{figure}

The set of feasible shifts is \( \Scal = \{ (s_1, s_2) \; | \; s_1 \leq - \beta , s_2 \leq \sqrt{-\beta - s_1} - \alpha \} \) on which the value function is \( \nu( s_1,s_2) = \max \{ -\sqrt{-\beta-s_1} , \alpha + s_2 \} \).
If \(\alpha < 0 \) \added{and \(\beta = 1\)}, the minimal norm shift is \(\sBar = (-1, 0)\) and \(\nu\) is nondifferentiable at that point since in a neighborhood of that point we have \(\nu( s_1,s_2 ) =  -\sqrt{-1-s_1} \). 
If \(\alpha > 0 \) \added{and \(\beta = 1\)}, in a neighborhood of \( \sBar\) (its expression is unwieldy) we have that \(\nu(s) = \alpha + s_2\) is subdifferentiable. 

\Cref{thm:level boundedness locally uniformly implies lsc} in the present work guarantees that the augmented Lagrangian converges to the solution set of the shifted problem, since its condition on recession directions is satisfied (indeed, there are no recession directions for the first constraint alone). In particular, if \((x^k)_{k\in\N}\) is a sequence generated by the inexact augmented Lagrangian method (\Cref{algo:IALM}), then the squared norm of the constraint violation \(  \left(\lfloor (x^k)^2 + \beta \rfloor_+ \right)^2 + \left(\lfloor x^k + \alpha \rfloor_+ \right)^2 \) is minimized at rate \( O\left(  \frac{1}{k} \right) \), and the objective values \(-x^k\) converge to the value of the closest feasible problem:
\begin{equation*} 
    \begin{array}{rl}
        \displaystyle \nu(\sBar) \triangleq \min_{x\in \X} & -x \\
        \text{s.t.} & x \in \argmin_{x' \in  \mathrm{dom}(f) } \left(\lfloor (x')^2 + \beta \rfloor_+ \right)^2 + \left(\lfloor x' + \alpha \rfloor_+ \right)^2 .
    \end{array}
\end{equation*}
In the case where subdifferentiability can be guaranteed, \cite{dai2023augmented} shows the existence of a subsequence of the augmented Lagrangian iterates that minimises some KKT conditions. 
\Cref{thm: nu subdifferentiable implies convergence} in the present work guarantees \( |-x^k - \nu(\sBar) | = O\left(  \frac{1}{\sqrt{k}} \right) \). 
\end{example}

\begin{example}[Second-order elliptic PDE] \label{example:second order elliptic PDE}.
This example illustrates an infinite-dimensional setting. Let \(\Omega\) be a closed, convex, bounded subset of a Hilbert space. Let \(L >0\) be a fixed positive real number.
Consider the second-order elliptic partial differential equation that consists of finding \(u \in \mathcal{L}(\Omega, L)\) (the set of \(L\)-Lipschitz functions on \(\Omega\)) that verifies
\begin{equation} \label{eq:second order elliptic PDE}
\left\{
\begin{aligned}
- \text{div} \left( \mathbf{A}(x) \nabla u(x) \right) + c(x) u(x) &= a(x) \quad \text{in } \Omega, \\
u(x) &= b(x) \quad \text{on } \mathrm{bdry}(\Omega),
\end{aligned}
\right.
\end{equation}
where \(\mathrm{bdry}(\Omega)\) is the boundary of \(\Omega\),  \(\mathbf{A}(x) \in C^1(\overline{\Omega})\) is uniformly elliptic and \(c(x) \in L^2(\Omega)\) is nonnegative over \(\Omega\), 

The variational formulation of \eqref{eq:second order elliptic PDE} is
\begin{equation} \label{eq:second order elliptic PDE Variational Version}
    \begin{array}{rl}
        \displaystyle \inf_{ v \in \mathcal{L}(\Omega, L) } & f(v)         \\[1mm]
        \text{s.t.}                    & v = b \text{ almost everywhere on } \mathrm{bdry}(\Omega),
    \end{array}
\end{equation}
where
\[
f(v) = \frac{1}{2} \int_{\Omega} \left[ \mathbf{A}(x) \nabla v(x) \cdot \nabla v(x) + c(x) v^2(x) - 2a(x)v(x) \right] dx.
\]

The IALM for this problem consists of the iterations:

\begin{align*}
    v^{k+1} & = \argmin_{ v \in \mathcal{L}(\Omega, L)} f(v) - \int_{\omega} \lambda^k(x) \left( v(x) - b(x) \right) dx + \tfrac{\gamma_k}{2}  \int_{\omega} \left( v(x) - b(x) \right)^2 dx \\
    \lambda^{k+1} &= \lambda^{k} - \gamma_k ( v- b ).
\end{align*}

\(f\) is level bounded due to the quadratic term and is therefore also level bounded with respect to constraint perturbations.
 \Cref{thm:level boundedness locally uniformly implies lsc} applies and we can conclude the algorithm converges semi-quantitatively to the solution set of the closest feasible problem
\begin{equation*}
        \begin{array}{rl}
            \displaystyle \inf_{ v \in \mathcal{L}(\Omega, L) } & f(v)         \\[1mm]
            \text{s.t.}                    & v \in \argmin_{w \in \mathcal{L}(\Omega, L)} \int_{\mathrm{bdry}(\Omega)} \left( w(x) - b(x) \right)^2 dx .
        \end{array}
\end{equation*}

This means that if it is impossible for an \(L\)-Lipschitz function to satisfy the boundary condition,
 then IALM will converge to the solution that minimizes the violation of the boundary conditions. 

\end{example}

%% file: conclusion.tex
\section{Conclusion}
\changed{
This work provides a comprehensive analysis of the inexact augmented Lagrangian method (IALM) applied to convex optimization problems that may lack feasible solutions. 
We also provide new results for the inexact proximal point algorithm (IPPA) applied to convex functions potentially lacking minimizers, which is of independent interest.}

\changed{For IPPA, we show that the subgradient-related terms converge to the minimal-norm subgradient and that the conjugate values \(h^*(s^k)\) converge to \(h^*(\sBar)\).
Finiteness or subdifferentiability of the dual function ensure convergence rates for these quantities.
}

\changed{
The function values and the constraint values robustly converge in IALM. 
Their limit corresponds to the value and constraint of the closest feasible problem under established sufficient conditions such as the absence of common recession directions of the objective and constraints, level boundedness of the objective with respect to constraint perturbations, or Lipschitz continuity of the objective together with polyhedral constraints. 
Some of these hypotheses also guarantee convergence rates for the function values and the constraint values.
We illustrated how these results can be applied in practical settings through examples.}

\changed{The presented results hold under standard assumptions on inexactness and step sizes, are applicable in infinite-dimensional Hilbert spaces, and offer a significantly clearer understanding of the ALM's behavior in degenerate or possibly infeasible scenarios.
This work thereby extends the reliability and applicability of augmented Lagrangian methods, providing stronger theoretical guarantees for their use in a wider array of practical optimization problems where feasibility or constraint qualification is not a given.
}


\section{Declarations}

\subsection{Funding}
R.A. is supported by the Agence d’Innovation Défense (AID).  
This work was also funded in part by the European Union (ERC grant CASPER, 101162889) and by the French government under the management of Agence Nationale de la Recherche (ANR) as part of the “Investissements d’avenir’’ and “France 2030’’ programmes (ANR-23-IACL-0008 “19-P3IA-0001, PR[AI]RIE-PSAI’’ 3IA Institute and ANR-22-CE33-0007, INEXACT).

\subsection{Competing interests}
The authors have no competing interests to declare.

%% file: appendix.tex
\appendix

\crefname{section}{appendix}{appendices}
\Crefname{section}{Appendix}{Appendices}



\section{The augmented Lagrangian method} \label{sec:ALM}
In this section, we provide a brief overview of the augmented Lagrangian method, following its historical development. We explain how the ALM was first established for equality-constrained convex optimization problems as a modification of the penalty method and then extended to the much broader class of convex optimization problems in the form \eqref{eq:convex_min_problem}.

\paragraph{The ALM for equality-constrained convex optimization problems.}
The augmented Lagrangian method was introduced as the "method of multipliers" in \cite{hestenes1969multiplier} and independently in \cite{powell1969method} for equality-constrained optimization problems. It was presented as an improvement on the penalty method.

The convex optimization problem with equality constraints is:
\begin{equation} \label{eq:equality_constrained_problem}
    \begin{array}{rl}
        \displaystyle \min_{x \in \X} & f(x)         \\[1mm]
        \text{s.t.}                    & c_i(x) = 0, \quad i = 1,\dots,m
    \end{array}
\end{equation}
where the \(c_i: \X \to \R\) are affine functions and \(f\) is convex, proper, and closed.
It is straightforward that \eqref{eq:equality_constrained_problem} is a special case of the formalism in this work \eqref{eq:convex_min_problem} when \(\KK = \{0\}^m\), and \( C: x \mapsto \left(c_1(x), \dots, c_m(x)\right) \).
 
The penalty method consists of solving the following sequence of problems:
\begin{equation} \label{eq:penalty_method}
    x^{k+1} = \argmin_{x \in \X} \left\{ f(x) + \frac{\gamma_k}{2} \sum_{i=1}^m  c_i(x)^2 \right\} = \argmin_{x \in \X} \left\{ f(x) + \frac{\gamma_k}{2} \|C(x)\|^2 \right\}
\end{equation}
where the positive penalty parameters \((\gamma_k)\) satisfy \(\gamma_k \to \infty\). As \(\gamma_k\) diverges, the constraints are satisfied asymptotically by the iterates \(x^k\), but the subproblem \eqref{eq:penalty_method} becomes increasingly ill-conditioned.

The method of multipliers modifies the penalty method by incorporating a linear term characterized by a vector of multipliers \(\lambda^k \in \R^m\), which is updated at each iteration:
\begin{equation} \label{eq:method_of_multipliers_equality_constraints}
    \begin{array}{rl}
        x^{k+1} &\in \argmin_{x \in \X}  f(x) + \sum_{i=1}^m \left[ -\lambda_i^k c_i(x) + \frac{\gamma_k}{2} c_i(x)^2 \right] \\
        \lambda_i^{k+1} &= \lambda_i^k - \gamma_k c_i(x^{k+1})  \quad  \text{for} \; i = 1, \dots , m
    \end{array}
\end{equation}
or equivalently in vectorized form:
\begin{equation} \label{eq:method_of_multipliers_equality_constraints_vectorized}
    \begin{array}{rl}
        x^{k+1}  &\in \argmin_{x \in \X}  f(x)  -\inner{\lambda^k}{C(x)} + \frac{\gamma_k}{2} \left\| C(x) \right\|^2 \\
        \lambda^{k+1} &= \lambda^k - \gamma_k C(x^{k+1}) 
    \end{array}
\end{equation}
Here, the dual variables \(\lambda^k\) are associated with the standard Lagrangian \(L_0(x,\lambda) = f(x) + \inner{\lambda}{C(x)}\). The penalty parameters \(\gamma_k > 0\) no longer need to diverge to infinity (they can be constant or updated dynamically), and the subproblems solved at each iteration generally exhibit better conditioning compared to the pure penalty method.

Subsequently, the method of multipliers was gradually renamed the augmented Lagrangian method in the optimization literature, since the quantity minimized at each iteration is the standard Lagrangian augmented with a quadratic penalty term.

\paragraph{The ALM for general convex optimization problems.}

More generally, \eqref{eq:convex_min_problem} encompasses a much broader class of constraints than the equality constraints in \eqref{eq:equality_constrained_problem}. Nevertheless, the augmented Lagrangian method can be derived naturally in this general setting.
We first rewrite Problem \eqref{eq:convex_min_problem} as an equivalent equality-constrained problem by introducing a slack variable \(s\):
\begin{equation} \label{eq:convex_min_problem_rewritten2}
    \begin{array}{rl}
        \displaystyle \inf_{x \in \X, s \in \Y} & f(x) + \delta_{\CC}(x,s)        \\[1mm]
        \text{s.t.}                    & s = 0 .
    \end{array}
\end{equation}
Here, the constraint \((x,s) \in \CC\) encodes the original problem structure: \(s = C(x) - y\) for some \(y \in \KK\). The objective uses the indicator function \(\delta_{\CC}\) of the set \(\CC\).
The augmented Lagrangian for this problem with penalty parameter \(\gamma > 0\) (associated with the constraint \(s=0\)) is
\begin{equation} \label{def:augmented lagrangian function}
    L_\gamma(x,s,\lambda) = f(x) + \delta_{ \CC}(x,s) - \inner{\lambda}{s}  + \tfrac{\gamma}{2} \left\Vert  s \right\Vert^2 \quad .
\end{equation}
The introduction of a slack variable to recast the problem into the setting of equality-constrained optimization is a common technique for deriving the ALM in more general settings; indeed, the ALM for problems with inequality constraints was derived using this approach in \cite{rockafellar1973dual}. 

Applying the method of multipliers \eqref{eq:method_of_multipliers_equality_constraints_vectorized} to the reformulated Problem \eqref{eq:convex_min_problem_rewritten2} (with objective \(f(x) + \delta_{\CC}(x,s)\) and constraint \(s=0\)) yields the iteration:
\begin{equation} \label{eq:augmented_lagrangian_method_general_problem}
    \begin{array}{rl}
        (x^{k+1} , s^{k+1} )  &\in \argmin_{(x,s) \in \CC}  f(x) -\inner{\lambda^k}{s} + \frac{\gamma_k}{2} \left\| s \right\|^2 \\
        \lambda^{k+1} &= \lambda^k - \gamma_k s^{k+1} .
    \end{array}
\end{equation}
Note that the minimization is over \((x,s) \in \CC\), which implicitly contains the \(\delta_{\CC}(x,s)\) term from \eqref{def:augmented lagrangian function}.

Through the change of variable \( y = C(x) - s \), noting that \((x,s) \in \CC\) is equivalent to \(x \in \X, y \in \KK\), we can rewrite \eqref{eq:augmented_lagrangian_method_general_problem} as:
\begin{equation} \label{eq:augmented_lagrangian_method_general_problem_xy}
    \begin{array}{rl}
        (x^{k+1} , y^{k+1} )  &\in \argmin_{x \in \X, y \in \KK}  f(x)  -\inner{\lambda^k}{C(x) - y} + \frac{\gamma_k}{2} \left\|C(x) - y \right\|^2 \\
        \lambda^{k+1} &= \lambda^k - \gamma_k  \left( C(x^{k+1}) - y^{k+1} \right) .
    \end{array}
\end{equation}

In the first line, the minimization with respect to \(y\) for a fixed \(x\) involves a quadratic function constrained to the closed convex set \(\KK\). The optimal \(y\) can be expressed analytically using the projection operator onto \(\KK\): \( y^*(x) = \Proj_{\KK}(C(x) - \lambda^k/\gamma_k) \). Substituting this back into the objective function and performing algebraic manipulations (while ignoring terms constant with respect to \(x\)), we can eliminate \(y\) and rewrite the update for \(x\) as:
\begin{equation} \label{eq:augmented_lagrangian_method_general_problem_x}
    \begin{array}{rl}
        x^{k+1}  &\in \argmin_{x \in \X}  f(x) + \frac{\gamma_k}{2} \left\|C(x) -\tfrac{\lambda^k}{\gamma_k} - \Proj_{\KK}\left(C(x) - \tfrac{\lambda^k}{\gamma_k}\right) \right\|^2 \\
        \lambda^{k+1} &= - \gamma_k  \left( C(x^{k+1}) - \tfrac{\lambda^k}{\gamma_k} - \Proj_{\KK}\left(C(x^{k+1}) - \tfrac{\lambda^k}{\gamma_k} \right) \right) .
    \end{array}
\end{equation}

If \(\KK\) is a closed convex cone, this formulation can be further simplified using the identity \(z - \Proj_\KK(z) = \Proj_{\KK^\circ}(z)\), where \(\KK^\circ\) is the polar cone of \(\KK\):
\begin{equation}     \label{eq:augmented_lagrangian_method_general_problem_x_polar_cone}
    \begin{array}{rl}
        x^{k+1}  &\in \argmin_{x \in \X}  f(x) + \frac{\gamma_k}{2} \left\| \Proj_{\KK^\circ}\left(C(x) - \tfrac{\lambda^k}{\gamma_k}\right) \right\|^2 \\
        \lambda^{k+1} &= - \gamma_k \Proj_{\KK^\circ} \left( C(x^{k+1}) - \tfrac{\lambda^k}{\gamma_k} \right) .
    \end{array}
\end{equation}
Notice in particular that we recover the augmented Lagrangian method for inequality-constrained optimization problems \eqref{eq:augmented_lagrangian_method_inequality_constraints} since the polar cone of \(\KK = \R_-^{m}\) is \(\KK^\circ = \R_+^{m}\).


In this work, we primarily use \eqref{eq:augmented_lagrangian_method_general_problem} involving \((x,s)\) for the augmented Lagrangian method, as it allows for simpler proofs relating IALM to the dual proximal point method. We study the more general inexact augmented Lagrangian method (IALM), which accounts for errors in solving the subproblems at each iteration, formally defined in \Cref{algo:IALM}.

\section{If the value function is not proper} \label{sec:case_where_the_shifted_problem_is_not_proper}

We briefly discuss the case where the value function \eqref{eq:convex_min_problem_shifted} is not proper, meaning that the function \( \nu: s \mapsto \inf_{x \in \X} f(x) + \delta_{(x,s)\in\CC}\) is not proper.
 It is clear that \(\nu\) cannot be identically \(+\infty\); indeed, taking any \(x \in \X\) such that \(f(x) < +\infty\) and any \(y\in \KK\) we have that \(\nu(C(x) - y)  \leq f(x) + \delta_{(x,C(x) -y)\in\CC} = f(x) < +\infty\).
 \changed{Therefore, \(\nu\) being improper would imply that there exists a shift \(\tilde s\) such that \(\nu(\tilde s) = -\infty\). We now show that in this case the first iterate of IALM instantly converges to \(-\infty\) in value.
 By definition of \(\nu(\tilde s) = -\infty\),} there exists a sequence \( (\tilde x^k)_{k\in\N} \subset \X \) such that  \( \forall k \in \N , \; ( \tilde x^k , \tilde s ) \in \CC \) and \( \lim_{k\to\infty} f(\tilde x^k) = -\infty \).
 Then for any \(\lambda^0 \in \Y\) we have that \( L_{\gamma_0}( \tilde x^k , \tilde s, \lambda^0) = f(\tilde x^k) + \delta_{(\tilde x^k, \tilde s)\in\CC} - \inner{\lambda^0}{\tilde s} + \tfrac{\gamma_0}{2} \left\| \tilde s \right\|^2 = f(\tilde x^k) - \inner{\lambda^0}{\tilde s} + \tfrac{\gamma_0}{2} \left\| \tilde s \right\|^2 \to - \infty\).
  This means that in the first step of IALM, the subproblem solved (consisting of minimizing \(L_{\gamma_0}(x,s,\lambda^0)\) over \((x,s) \in \CC\)) has for value \(-\infty\).
   In this case the algorithm converges in a single iteration in value to \(-\infty\).

%% file: references.bib
@article{reich1977infinite,
  title        = {On infinite products of resolvents},
  author       = {Reich, Simeon},
  journal      = {Atti della Accademia Nazionale dei Lincei. Classe di Scienze Fisiche, Matematiche e Naturali. Rendiconti},
  volume       = {63},
  number       = {5},
  pages        = {338--340},
  year         = {1977}
}

@book{rockafellar2009variational,
  title        = {Variational Analysis},
  author       = {Rockafellar, R. Tyrrell and Wets, Roger J.-B.},
  series       = {Grundlehren der mathematischen Wissenschaften},
  volume       = {317},
  publisher    = {Springer},
  year         = {1998},
  doi          = {10.1007/978-3-642-02431-3},
  note         = {3rd printing (2009) includes corrections}
}

@article{chiche2016augmented,
  title        = {How the augmented {L}agrangian algorithm can deal with an infeasible convex quadratic optimization problem},
  author       = {Chiche, Alice and Gilbert, Jean-Charles},
  journal      = {Journal of Convex Analysis},
  volume       = {23},
  number       = {2},
  pages        = {425--459},
  year         = {2016}
}

@article{dai2023augmented,
  title        = {The augmented {L}agrangian method can approximately solve convex optimization with least constraint violation},
  author       = {Dai, Yu-Hong and Zhang, Liwei},
  journal      = {Mathematical Programming},
  series       = {B},
  volume       = {200},
  number       = {2},
  pages        = {633--667},
  year         = {2023},
  doi          = {10.1007/s10107-022-01843-2},
  publisher    = {Springer}
}

@article{birgin2015optimality,
  title        = {Optimality properties of an augmented {L}agrangian method on infeasible problems},
  author       = {Birgin, Ernesto G. and Mart{\'\i}nez, Jos{\'e} Mario and Prudente, Leandro da Fonseca},
  journal      = {Computational Optimization and Applications},
  volume       = {60},
  number       = {3},
  pages        = {609--631},
  year         = {2015},
  doi          = {10.1007/s10589-014-9685-5},
  publisher    = {Springer}
}

@article{rockafellar1976augmented,
  title        = {Augmented {L}agrangians and applications of the proximal point algorithm in convex programming},
  author       = {Rockafellar, R. Tyrrell},
  journal      = {Mathematics of Operations Research},
  volume       = {1},
  number       = {2},
  pages        = {97--116},
  year         = {1976},
  doi          = {10.1287/moor.1.2.97},
  publisher    = {INFORMS}
}

@book{rockafellar1970convex,
  title        = {Convex Analysis},
  author       = {Rockafellar, R. Tyrrell},
  publisher    = {Princeton University Press},
  year         = {1970},
  doi          = {10.1515/9781400873173},
  note         = {DOI corresponds to the electronic edition}
}

@article{gonccalves2015augmented,
  title        = {Augmented {L}agrangian methods for nonlinear programming with possible infeasibility},
  author       = {Gon{\c{c}}alves, Max L. N. and Melo, Jefferson G. and Prudente, Leandro F.},
  journal      = {Journal of Global Optimization},
  volume       = {63},
  number       = {2},
  pages        = {297--318},
  year         = {2015},
  doi          = {10.1007/s10898-015-0289-0},
  publisher    = {Springer}
}

@incollection{rockafellar1971saddle,
  title        = {Saddle-points and convex analysis},
  author       = {Rockafellar, R. Tyrrell},
  booktitle    = {Differential Games and Related Topics},
  editor       = {Kuhn, H. W. and Szeg{\"o}, G. P.},
  pages        = {109--128},
  publisher    = {North-Holland},
  year         = {1971}
}

@article{drori2014performance,
  title        = {Performance of first-order methods for smooth convex minimization: a novel approach},
  author       = {Drori, Yoel and Teboulle, Marc},
  journal      = {Mathematical Programming},
  volume       = {145},
  number       = {1--2},
  pages        = {451--482},
  year         = {2014},
  doi          = {10.1007/s10107-013-0653-0},
  publisher    = {Springer}
}

@article{taylor2017exact,
  title        = {Exact worst-case performance of first-order methods for composite convex optimization},
  author       = {Taylor, Adrien B. and Hendrickx, Julien M. and Glineur, Fran{\c{c}}ois},
  journal      = {SIAM Journal on Optimization},
  volume       = {27},
  number       = {3},
  pages        = {1283--1313},
  year         = {2017},
  doi          = {10.1137/16M108104X},
  publisher    = {SIAM}
}

@article{hermans2022qpalm,
  title        = {QPALM: A proximal augmented {L}agrangian method for nonconvex quadratic programs},
  author       = {Hermans, Ben and Themelis, Andreas and Patrinos, Panagiotis},
  journal      = {Mathematical Programming Computation},
  volume       = {14},
  number       = {3},
  pages        = {497--541},
  year         = {2022},
  doi          = {10.1007/s12532-022-00218-0},
  publisher    = {Springer}
}

@article{armand2019augmented,
  title        = {An augmented {L}agrangian method for equality constrained optimization with rapid infeasibility detection capabilities},
  author       = {Armand, Paul and Tran, Ngoc Nguyen},
  journal      = {Journal of Optimization Theory and Applications},
  volume       = {181},
  number       = {1},
  pages        = {197--215},
  year         = {2019},
  doi          = {10.1007/s10957-018-1401-7},
  publisher    = {Springer}
}

@article{hestenes1969multiplier,
  title        = {Multiplier and gradient methods},
  author       = {Hestenes, Magnus R.},
  journal      = {Journal of Optimization Theory and Applications},
  volume       = {4},
  number       = {5},
  pages        = {303--320},
  year         = {1969},
  doi          = {10.1007/BF00927673},
  publisher    = {Springer}
}

@incollection{powell1969method,
  title        = {A method for nonlinear constraints in minimization problems},
  author       = {Powell, Michael J. D.},
  booktitle    = {Optimization},
  editor       = {Fletcher, Roger},
  pages        = {283--298},
  publisher    = {Academic Press},
  address      = {New York, NY},
  year         = {1969}
}

@article{rockafellar1973dual,
  title        = {A dual approach to solving nonlinear programming problems by unconstrained optimization},
  author       = {Rockafellar, R. Tyrrell},
  journal      = {Mathematical Programming},
  volume       = {5},
  number       = {1},
  pages        = {354--373},
  year         = {1973},
  doi          = {10.1007/BF01580138},
  publisher    = {Springer}
}

@book{bertsekas1982constrained,
  title        = {Constrained Optimization and Lagrange Multiplier Methods},
  author       = {Bertsekas, Dimitri P.},
  publisher    = {Academic Press},
  year         = {1982},
  doi          = {10.1016/C2013-0-10366-2},
  isbn         = {978-0-12-093480-5}
}

@article{solodov2001unified,
  title        = {A unified framework for some inexact proximal point algorithms},
  author       = {Solodov, Mikhail V. and Svaiter, Benar F.},
  journal      = {Numerical Functional Analysis and Optimization},
  volume       = {22},
  number       = {7--8},
  pages        = {1013--1035},
  year         = {2001},
  doi          = {10.1081/NFA-100108320},
  publisher    = {Taylor \& Francis}
}

@incollection{lemaire1992convergence,
  title        = {About the convergence of the proximal method},
  author       = {Lemaire, Bernard},
  booktitle    = {Advances in Optimization},
  editor       = {Oettli, Werner and Pallaschke, Diethard},
  series       = {Lecture Notes in Economics and Mathematical Systems},
  volume       = {382},
  pages        = {39--51},
  publisher    = {Springer},
  year         = {1992},
  doi          = {10.1007/978-3-642-51682-5_4}
}

@inproceedings{agrawal2019differentiable,
  title        = {Differentiable convex optimization layers},
  author       = {Agrawal, Akshay and Amos, Brandon and Barratt, Shane and Boyd, Stephen and Diamond, Steven and Kolter, J. Zico},
  booktitle    = {Advances in Neural Information Processing Systems},
  volume       = {32},
  pages        = {9562--9574},
  year         = {2019},
  url          = {https://papers.nips.cc/paper/9152-differentiable-convex-optimization-layers}
}

@inproceedings{bambade2024leveraging,
  title        = {Leveraging augmented-{L}agrangian techniques for differentiating over infeasible quadratic programs in machine learning},
  author       = {Bambade, Antoine and Schramm, Fabian and Taylor, Adrien B. and Carpentier, Justin},
  booktitle    = {International Conference on Learning Representations},
  year         = {2024},
  url          = {https://openreview.net/forum?id=YCPDFfmkFr},
  note         = {ICLR 2024 (Spotlight)}
}

@inproceedings{amos2017optnet,
  title        = {OptNet: Differentiable optimization as a layer in neural networks},
  author       = {Amos, Brandon and Kolter, J. Zico},
  booktitle    = {Proceedings of the 34th International Conference on Machine Learning},
  series       = {Proceedings of Machine Learning Research},
  volume       = {70},
  pages        = {136--145},
  year         = {2017},
  publisher    = {PMLR},
  url          = {https://proceedings.mlr.press/v70/amos17a.html}
}

@article{scokaert1999feasibility,
  title        = {Feasibility issues in linear model predictive control},
  author       = {Scokaert, Pierre O. M. and Rawlings, James B.},
  journal      = {AIChE Journal},
  volume       = {45},
  number       = {8},
  pages        = {1649--1659},
  year         = {1999},
  doi          = {10.1002/aic.690450805},
  publisher    = {Wiley}
}

@article{martinet1970breve,
  title        = {Br{\`e}ve communication. {R}{\'e}gularisation d'in{\'e}quations variationnelles par approximations successives},
  author       = {Martinet, Bernard},
  journal      = {Revue fran{\c{c}}aise d'informatique et de recherche op{\'e}rationnelle. S{\'e}rie rouge},
  volume       = {4},
  number       = {R3},
  pages        = {154--158},
  year         = {1970},
  doi          = {10.1051/m2an/197004R301541}
}

@incollection{gauvin1982differential,
  title        = {Differential properties of the marginal function in mathematical programming},
  author       = {Gauvin, Jacques and Dubeau, Fran{\c{c}}ois},
  booktitle    = {Optimality and Stability in Mathematical Programming},
  editor       = {Guignard, Monique},
  series       = {Mathematical Programming Studies},
  volume       = {19},
  pages        = {101--119},
  publisher    = {Springer},
  year         = {1982},
  doi          = {10.1007/BFb0120984}
}

@book{bonnans2000perturbation,
  title        = {Perturbation Analysis of Optimization Problems},
  author       = {Bonnans, J. Fr{\'e}d{\'e}ric and Shapiro, Alexander},
  publisher    = {Springer},
  year         = {2000},
  doi          = {10.1007/978-1-4612-1394-9}
}

@article{hogan1973directional,
  title        = {Directional derivatives for extremal-value functions with applications to the completely convex case},
  author       = {Hogan, William},
  journal      = {Operations Research},
  volume       = {21},
  number       = {1},
  pages        = {188--209},
  year         = {1973},
  doi          = {10.1287/opre.21.1.188},
  publisher    = {INFORMS}
}

@book{rockafellar1974conjugate,
  title        = {Conjugate Duality and Optimization},
  author       = {Rockafellar, R. Tyrrell},
  series       = {CBMS-NSF Regional Conference Series in Applied Mathematics},
  volume       = {16},
  publisher    = {SIAM},
  year         = {1974},
  doi          = {10.1137/1.9781611970524}
}

@article{brondsted1965subdifferentiability,
  title        = {On the subdifferentiability of convex functions},
  author       = {Br{\o}ndsted, Arne and Rockafellar, Ralph Tyrrell},
  journal      = {Proceedings of the American Mathematical Society},
  volume       = {16},
  number       = {4},
  pages        = {605--611},
  year         = {1965},
  doi          = {10.1090/S0002-9939-1965-0178103-8},
  publisher    = {American Mathematical Society}
}

@article{cobzas1978norm,
  title        = {Norm-preserving extension of convex {L}ipschitz functions},
  author       = {Cobza{\c{s}}, {\c{S}}tefan and Must{\u{a}}{\c{t}}{\u{a}}, Costic{\u{a}}},
  journal      = {Journal of Approximation Theory},
  volume       = {24},
  number       = {3},
  pages        = {236--244},
  year         = {1978},
  doi          = {10.1016/0021-9045(78)90028-X},
  publisher    = {Elsevier}
}
